\documentclass[reqno, 11pt, a4paper]{amsart}
\usepackage{latexsym,ifthen,xspace}
\usepackage{amsmath,amssymb,amsthm}
\usepackage{enumerate, calc, bbm}
\usepackage{paralist}
\usepackage[colorlinks=true, pdfstartview=FitV, linkcolor=blue,
  citecolor=blue, urlcolor=blue, pagebackref=false]{hyperref}
\usepackage[text={33pc,605pt},centering]{geometry}    
\usepackage{graphicx}
\newtheorem{theorem}{Theorem}[section]
\newtheorem{lemma}[theorem]{Lemma}
\newtheorem{prop}[theorem]{Proposition}
\newtheorem{assumption}[theorem]{Assumption}

\theoremstyle{definition}
\newtheorem{definition}[theorem]{Definition}
\newtheorem{example}[theorem]{Example}

\theoremstyle{remark}
\newtheorem{remark}[theorem]{Remark}

\numberwithin{equation}{section}

\DeclareMathAlphabet{\mathsl}{OT1}{cmss}{m}{sl}
\SetMathAlphabet{\mathsl}{bold}{OT1}{cmss}{bx}{sl}

\newcommand{\al}{\ensuremath{\alpha}}
\newcommand{\be}{\ensuremath{\beta}}
\newcommand{\ga}{\ensuremath{\gamma}}
\newcommand{\de}{\ensuremath{\delta}}

\newcommand{\ze}{\ensuremath{\zeta}}

\renewcommand{\th}{\ensuremath{\theta}}

\newcommand{\ka}{\ensuremath{\kappa}}

\newcommand{\si}{\ensuremath{\sigma}}

\newcommand{\om}{\ensuremath{\omega}}

\newcommand{\ve}{\ensuremath{\varepsilon}}

\newcommand{\vp}{\ensuremath{\varphi}}

\newcommand{\Si}{\ensuremath{\Sigma}}

\newcommand{\Om}{\ensuremath{\Omega}}
\newcommand{\cB}{\ensuremath{\mathcal B}}
\newcommand{\cC}{\ensuremath{\mathcal C}}

\newcommand{\cF}{\ensuremath{\mathcal F}}

\newcommand{\cL}{\ensuremath{\mathcal L}}

\newcommand{\cO}{\ensuremath{\mathcal O}}

\newcommand{\cS}{\ensuremath{\mathcal S}}

\newcommand{\bbN}{\ensuremath{\mathbb N}} 
 
\newcommand{\bbP}{\ensuremath{\mathbb P}} 
 
\newcommand{\bbR}{\ensuremath{\mathbb R}}

\newcommand{\bbZ}{\ensuremath{\mathbb Z}} 
\newcommand{\md}{\ensuremath{\mathrm{d}}}
\newcommand{\mD}{\ensuremath{\mathrm{D}}}

\newcommand{\norm}[3]{%
   \ensuremath{%
     \mathchoice{\big\lVert #1 \big\rVert}
     {\lVert #1 \rVert}
     {\lVert #1 \rVert}
     {\lVert #1 \rVert}_{\raisebox{-.0ex}{$\scriptstyle \ell^{\raisebox{.2ex}{$\scriptscriptstyle #2$}} (#3)$}}
   }
}

\newcommand{\Norm}[2]{%
  \ensuremath{%
    \mathchoice{\big\lVert #1 \big\rVert}
     {\lVert #1 \rVert}
     {\lVert #1 \rVert}
     {\lVert #1 \rVert}_{\raisebox{-.0ex}{$\scriptstyle #2$}}
  }
}

\DeclareMathOperator{\mean}{\mathbb{E}}
\DeclareMathOperator{\Mean}{\mathrm{E}}
\DeclareMathOperator{\prob}{\mathbb{P}} 
\DeclareMathOperator{\Prob}{\mathrm{P}} 

\DeclareMathOperator{\supp}{\mathrm{supp}}

\newcommand{\ldef}{\ensuremath{\mathrel{\mathop:}=}}
\newcommand{\rdef}{\ensuremath{=\mathrel{\mathop:}}}

\newcommand{\overbar}[1]{\mkern 2mu\overline{\mkern-3mu#1\mkern-1mu}\mkern 2mu}

\newcommand{\indicator}{\mathbbm{1}}

\begin{document}

\title[QFCLT for the RCM with ergodic conductances on random graphs]{Quenched invariance principles for the random conductance model on a random graph with degenerate ergodic weights}


\author{Jean-Dominique Deuschel}
\address{Technische Universit\"at Berlin}
\curraddr{Strasse des 17. Juni 136, 10623 Berlin}
\email{deuschel@math.tu-berlin.de}
\thanks{}

\author{Tuan Anh Nguyen}
\address{Technische Universit\"at Berlin}
\curraddr{Strasse des 17. Juni 136, 10623 Berlin}
\email{tanguyen@math.tu-berlin.de}
\thanks{}

\author{Martin Slowik}
\address{Technische Universit\"at Berlin}
\curraddr{Strasse des 17. Juni 136, 10623 Berlin}
\email{slowik@math.tu-berlin.de}
\thanks{}



\date{\today}

\dedicatory{}

\begin{abstract}
  We consider a stationary and ergodic random field $\{\om(e) : e \in E_d\}$ that is parameterized by the edge set of the Euclidean lattice $\bbZ^d$, $d \geq 2$.  The random variable $\om(e)$, taking values in $[0, \infty)$ and satisfying certain moment bounds, is thought of as the conductance of the edge $e$.  Assuming that the set of edges with positive conductances give rise to a unique infinite cluster $\cC_{\infty}(\om)$, we prove a quenched invariance principle for the continuous-time random walk among random conductances under under relatively mild conditions on the structure of the infinite cluster.  An essential ingredient of our proof is a new anchored relative isoperimetric inequality. 
\end{abstract}

\maketitle

\tableofcontents

\section{Introduction}\label{sec:INTRO}
\subsection{The model}
Consider the $d$-dimensional Euclidean lattice, $(\bbZ^d, E_d)$, for $d \geq 2$, where the edge set, $E_d$, is given by the set of all non-oriented nearest neighbor bonds.  Let $(\Om, \cF) = ([0, \infty)^{E_d}, \cB([0, \infty))^{\otimes E_d})$ be a measurable space equipped with the Borel-$\si$-algebra.  For $\om \in \Om$, we refer to $\om(\{x,y\})$ as the \emph{conductance} of the corresponding edge $\{x,y\}$.  Henceforth, we consider a probability measure $\prob$ on $(\Om, \cF)$, and we write $\mean$ to denote the expectation with respect to $\prob$.  Further, a \emph{translation} or \emph{shift} by $z \in \bbZ^d$ is a map $\tau_z\!: \Om \to \Om$,
\begin{align}\label{eq:def:translation}
  (\tau_z \om)(\{x,y\})
  \;\ldef\;
  \om(\{x+z,y+z\}),
  \qquad \{x,y\}\in E_d.
\end{align}
The set $\{\tau_x : x \in \bbZ^d\}$ together with the operation $\tau_x \circ \tau_y \ldef \tau_{x+y}$ defines the group of space shifts.

For any $\om \in \Om$, the induced set of \emph{open} edges is denoted by
\begin{align*}
  \cO
  \;\equiv\;
  \cO(\om)
  \;\ldef\;
  \big\{ e \in E_d \mid \om(e) > 0\}
  \;\subset\;
  E_d.
\end{align*}
We also write $x \sim y$ if $\{x, y\} \in \cO(\om)$.  Further, we denote by $\cC_{\infty}(\om)$ the subset of vertices of $\bbZ^d$ that are in infinite connected components.

Throughout the paper, we will impose assumptions both on the law $\prob$ and on geometric properties of the infinite cluster.
\begin{assumption}\label{assumption:law}
  Assume that $\prob$ satisfies the following conditions:
  \begin{enumerate}[(i)]
  \item The law $\prob$ is stationary and ergodic with respect to translations of $\bbZ^d$.
  \item $\mean[\om(e)] < \infty$ for all $e \in E_d$.
  \item For $\prob$-a.e.\ $\om$, the set $\cC_{\infty}(\om)$ is connected, i.e.\ there exists a unique infinite connected component -- also called infinite open cluster -- and $\prob[0 \in \cC_{\infty}] > 0$.
    \end{enumerate}
\end{assumption}
Let $\Om_0 = \big\{\om \in \Om \,:\, 0 \in \cC_{\infty}(\om) \big\}$ and introduce the conditional measure
\begin{align}
  \prob_0[\,\cdot\,] \;\ldef\; \prob[\,\cdot\, |\, 0 \in \cC_{\infty}],
\end{align}
and we write $\mean_0$ to denote the expectation with respect to $\prob_0$.  We denote by $d^{\om}$ the natural graph distance on $(\cC_{\infty}(\om), \cO(\om))$, in the sense that for any $x, y \in \cC_{\infty}(\om)$, $d^{\om}(x,y)$ is the minimal length of a path between $x$ and $y$ that consists only of edges in $\cO(\om)$.  For $x \in \cC_{\infty}(\om)$ and $r \geq 0$, let $B^{\om}(x,r) \ldef {\{ y \in \cC_{\infty}(\om) : d^{\om}(x, y) \leq \lfloor r \rfloor\}}$ be the closed ball with center $x$ and radius $r$ with respect to $d^{\om}$, and we write $B(x,r) \ldef \{y \in \bbZ^d \,:\, |y-x|_1 \leq \lfloor r \rfloor\}$ for the corresponding closed ball with respect to the $\ell^1$-distance on $\bbZ^d$.  Further, for a given subset $B \subset \bbZ^d$ we denote by $|B|$ the cardinality of $B$, and we define the \emph{relative} boundary of $A \subset B$ by
\begin{align*}
  \partial_{\!B}^{\om} A
  \;\ldef\;
  \big\{
    \{x,y\} \in \cO(\om)
    \;:\;
    x \in A \,\text{ and }\, y \in B \setminus A
  \big\}
\end{align*}
and we simply write $\partial^{\om} A$ if $B \equiv \cC_{\infty}(\om)$.  The corresponding boundary on $(\bbZ^d, E_d)$ is denoted by $\partial_B A$ and $\partial A$, respectively.
\begin{definition}[regular balls]\label{def:regular}
   Let $C_{\mathrm{V}} \in (0, 1]$, $C_{\mathrm{riso}} \in (0, \infty)$ and $C_{\mathrm{W}} \in [1, \infty)$ be fixed constants.  For $x \in \cC_{\infty}(\om)$ and $n \geq 1$, we say a ball $B^{\om}(x, n)$ is \emph{regular} if it satisfies the following conditions:
   \begin{enumerate}[i)]
   \item volume regularity of order $d$:
     \begin{align}\label{eq:ass:balls}
       C_{\mathrm{V}}\, n^d \;\leq\; |B^{\om}(x, n)|
     \end{align}
   \item (weak) relative isoperimetric inequality: There exists $\cS^{\om}(x, n) \subset \cC_{\infty}(\om)$ connected such that $B^{\om}(x, n) \subset \cS^{\om}(x, n) \subset B^{\om}(x, C_{\mathrm{W}} n)$ and
     \begin{align}\label{eq:ass:riso}
       |\partial_{\cS^{\om}(x, n)}^{\om} A|
       \;\geq\;
       C_{\mathrm{riso}}\, n^{-1}\, |A|
     \end{align}
     for every $A \subset \cS^{\om}(x, n)$ with $|A| \leq \tfrac{1}{2}\, |\cS^{\om}(x, n)|$.
   \end{enumerate}
 \end{definition}
\begin{assumption}[$\th$-very regular balls]\label{assumption:cluster}
  For some $\th \in (0,1)$ assume that for $\prob_0$-a.e.\ $\om$ there exists $N_0(\om) < \infty$ such that for all $n \geq N_0(\om)$ the ball $B^{\om}(0, n)$ is $\th$-\emph{very regular}, that is, the ball $B^{\om}(x, r)$ is regular for every $x \in B^{\om}(0, n)$ and $r \geq n^{\th/d}$. 
\end{assumption}
\begin{remark}
  (i) The Euclidean lattice $(\bbZ^d, E_d)$ satisfies the assumption above with $\th = 0$.\\
  (ii) The notion of $\th$-very regular balls is particularly useful in the context of random graphs, e.g.\ supercritical Bernoulli percolation clusters \cite{Ba04} or clusters in percolation models with long range correlations \cite{Sa14} (see the examples below for more details).  Such random graphs have typically a local irregular behaviour, in the sense that the conditions of volume growth and relative isoperimetric inequality fail on small scales.  Roughly speaking, Assumption~\ref{assumption:cluster} provides a uniform lower bound on the radius of regular balls.\\
  (iii) In contrast to the (weak) relative isoperimetric inequality \eqref{eq:ass:riso}, the (standard) isoperimetric inequality on $\bbZ^d$ reads
  \begin{align}\label{eq:iso:Zd}
    |\partial^{\om} A|
    \;\geq\;
    C_{\mathrm{iso}}\, |A|^{(d-1)/d},
    \qquad \forall\, A \subset \bbZ^d.
  \end{align}
  On random graphs, however, such an inequality is true only for large enough sets.  However, under the assumption that the ball $B^{\om}(x,n)$ is $\th$-very regular, the isoperimetric inequality \eqref{eq:iso:Zd} holds for all $A \subset B^{\om}(x,n)$ with $|A| > n^{\th}$; cf.\ Lemma~\ref{lemma:iso:large_sets:random} below.
\end{remark}
For any fixed realization $\om \in \Om$, we are interested in a continuous-time Markov chain, $X = \{X_t : t \geq 0\}$, on $\cC_{\infty}(\om)$.  We refer to $X$ as \emph{random walk among random conductances} or \emph{random conductance model} (RCM).  Set $\mu^{\om}(x) = \sum_{y \sim x} \om(\{x,y\})$, $X$ is the process that waits at the vertex $x \in \cC_{\infty}(\om)$ an exponential time with mean $1/\mu^{\om}(x)$ and then jumps to a vertex $y$ that is connected to $x$ by an open edge with probability $\om(\{x,y\}) / \mu^{\om}(x)$.  Since the holding times are space dependent, this process is also called \emph{variable speed random walk} (VSRW).  The process $X$ is a Markov process with generator, $\cL^{\om}$, acting on bounded functions as
\begin{align}
  \big(\cL^{\om} f\big)(x)
  \;=\;
  \sum_{y \in \bbZ^d} \om(\{x,y\})\, \big( f(y) - f(x) \big),
  \qquad x \in \cC_{\infty}(\om).
\end{align}
We denote by $\Prob_x^{\om}$ the quenched law of the process starting at the vertex $x \in \cC_{\infty}(\om)$.  The corresponding expectation will be denoted by $\Mean_x^{\om}$.  Notice that $X$ is a \emph{reversible} Markov chain with respect to the counting measure. 
\subsection{Main result}
We are interested in the long time behavior of the random walk among random conductances for $\prob_0$-almost every realization $\om$.  In particular, we are aiming at obtaining a quenched functional central limit theorem (QFCLT) for the process $X$ in the following sense.
\begin{definition}
  Set $X_t^{(n)} \ldef \frac{1}{n} X_{t n^2}$, $t \geq 0$.  We say that the \emph{quenched functional CLT} or \emph{quenched invariance principle} holds for $X$, if for every $T>0$ and every bounded continuous function $F$ on the Skorohod space $D([0, T], \bbR^d)$, it holds that $\Mean_0^{\om}[F(X^{(n)})] \to \Mean_0^{\mathrm{BM}}[F(\Si \cdot W)]$ as $n \to \infty$ for $\prob_0$-a.e.\ $\om$, where $(W, \Prob_0^{\mathrm{BM}})$ is a Brownian motion on $\bbR^d$ starting at 0 with covariance matrix $\Si^2 = \Si \cdot \Si^T$.
\end{definition}
Our main result relies on the following integrability condition.
\begin{assumption}[Integrability condition]\label{assumption:pq}
  For some $p, q \in [1, \infty]$ and $\th \in (0, 1)$ with
  \begin{align}
    \frac{1}{p} \,+\, \frac{1}{q}
    \;<\;
    \frac{2(1-\th)}{d-\th},
  \end{align}
  assume that the following integrability condition holds
  \begin{align}
    \mean\big[\om(e)^p\big] \;<\; \infty
    \qquad \text{and} \qquad
    \mean\big[\om(e)^{-q} \indicator_{e \in \cO}\big] \;<\; \infty,
  \end{align}
  where we used the convention that $0/0 = 0$.
\end{assumption}
\begin{theorem}[Quenched invariance principle]\label{thm:QIP:X}
  For $d \geq 2$ suppose that $\th \in (0, 1)$ and $p, q \in [1, \infty]$ satisfy Assumptions~\ref{assumption:law}, \ref{assumption:cluster} and \ref{assumption:pq}.  Then, the QFCLT holds for the process $X$ with a deterministic and non-degenerate covariance matrix $\Si^2$.
\end{theorem}
\begin{remark}
  If the law $\prob$ of the conductances is invariant under reflection and rotation of $\bbZ^d$ by $\pi/2$, the limiting Brownian motion is isotropic in the sense that its covariance matrix $\Si^2$ is of the form $\Si^2 = \si^2 I$ for some $\si > 0$. (Here $I \in \bbR^{d \times d}$ denotes the identity matrix.)
\end{remark}
\begin{remark}
  Consider the Markov process $Y = \{Y_t : t \geq 0\}$ on $\cC_{\infty}(\om)$ that waits at the vertex $x \in \cC_{\infty}(\om)$ an exponential time with mean $1$ and then jumps to a neighboring vertex $y$ with probability $\om(\{x,y\})/\mu^{\om}(x)$.  This process is also called \emph{constant speed random walk} (CSRW).  Notice that the process $Y$ can be obtained from the process $X$ by a time change, that is $Y_t \ldef X_{a_t}$ for $t \geq 0$, where $a_t \ldef \inf\{s \geq 0 : A_s > t\}$ denotes the right continuous inverse of the functional
  \begin{align*}
    A_t 
    \;\ldef\;
    \int_0^t \mu^{\om}(X_s)\, \md s,
    \qquad t \geq 0.
  \end{align*}
  By the ergodic theorem and Lemma~\ref{lemma:ENVI}, we have that $\lim_{t \to \infty} A_t / t = \mean_0[\mu^{\om}(0)]$ for $\prob_0$-a.e\ $\om$.  Hence, under the assumptions of Theorem~\ref{thm:QIP:X}, the rescaled process $Y$ converges to a Brownian motion on $\bbR^d$ with deterministic and non-degenerate covariance matrix $\Si_Y^2 = \mean_0[\mu^{\om}(0)]^{-1} \Si^2$, see \cite[Section~6.2]{ABDH13}.
\end{remark}
\begin{remark}
  Notice that Assumption~\ref{assumption:law} and the remark above implies that $\prob_0$-a.s.\ the process $X$ does not explode in finite time.
\end{remark}
Random walks among random conductances is one of the most studied examples of random walks in random environments.  Since the pioneering works of De Masi, Ferrari, Goldstein and Wick \cite{dMFGW89} and Kipnis and Varadhan \cite{KV86} which proved a \emph{weak FCLT} for stationary and ergodic laws $\prob$ with $\mean[\om(e)] < \infty$, in the last two decades much attention has been devoted to obtain a \emph{quenched} FCLT.

For i.i.d.\ environments ($\prob$ is a product measure), it turns out that no moment conditions are required.  Based on the previous works by Mathieu \cite{Ma08}, Biskup and Prescott \cite{BP07}, Barlow and Deuschel \cite{BD10} (for similar results for simple random walks on supercritical Bernoulli percolation clusters see also Sidoravious and Sznitman \cite{SS04}, Berger and Biskup \cite{BB07}, Mathieu and Piatniski \cite{MP07}) it has been finally shown by Andres, Barlow, Deuschel and Hambly \cite{ABDH13} that a QFCLT for i.i.d.\ environments holds provided that $\prob_0[\om(e) > 0] > p_c$ with $p_c \equiv p_c(d)$ being the bond percolation threshold.  Recently, Procaccia, Rosenthal and Sapozhnikov \cite{PRS13} have studied a quenched invariance principle for simple random walks on a certain class of percolation models with long range correlations including random interlacements and level sets of the Gaussian Free Field (both in $d \geq 3$).

For general ergodic, elliptic environments, $\prob[0 < \om(e) < \infty] = 1$, where the infinite connected component $\cC_{\infty}(\om)$ coincides with $\bbZ^d$, the first moment condition on the conductances, $\mean[\om(e)] < \infty$ and $\mean[\om(e)^{-1}] < \infty$, is necessary for a QFCLT to hold, see Barlow, Burdzy and Tim\'ar \cite{BBT16, BBT15}.  The uniformly elliptic situation, treated by  Boivin \cite{Bo93}, Sidoravious and Sznitman \cite{SS04} (cf.\ Theorem~1.1 and Remark~1.3 therein), Barlow and Deuschel \cite{BD10}, has been relaxed by Andres, Deuschel and Slowik \cite{ADS15} to the condition in Assumption~\ref{assumption:pq} with $\th = 0$.  As it turned out, for the constant speed random walk $Y$ as defined above, this moment condition is optimal for a \emph{quenched local limit theorem} to hold, see \cite{ADS16a}.  In dimension $d = 2$, Biskup proved a QFCLT under the (optimal) first moment condition, and it is an open problem if this remains true in dimensions $d \geq 3$.

In this paper, we are interested in the random conductance model beyond the elliptic setting. We prove a quenched invariance principle in the case of stationary and ergodic laws under mild assumptions on geometric properties of the resulting clusters and on the integrability of $\prob$.  This framework includes the models considered in \cite{ADS15} and \cite{PRS13}.  The main novelty is a new anchored relative isoperimetric inequality (Lemma~\ref{lemma:riso:weight}) that is used to show in a robust way the $\ell^1$-sublinearity of the corrector (for more details see below).  Another important aspect is that neither an a priori knowledge on the distribution of the size of holes in the connected components nor on properties of the chemical distance is needed.  In particular, our proof does not rely on the directional sublinearity of the corrector.
\medskip

In the sequel, we give a brief list of motivating examples of probability measures on $[0, \infty)^{E_d}$ for the conductances.
\begin{example}[Supercritical Bernoulli percolation cluster]
  Consider a supercritical Bernoulli bond percolation $\{\om(e) : e \in E_d\}$, that is, $\om(e) \in \{0, 1\}$ are i.i.d.\ random variables with $\prob[\omega(e) = 1] > p_{\mathrm{c}}$. The almost sure existence of a unique infinite cluster is guaranteed by Burton--Keane's theorem, while Assumption~\ref{assumption:cluster} on $\theta$-very regular balls for any $\th \in (0, 1)$ follows from a series of results in \cite{Ba04}: Theorem~2.18 a), c) together with Lemma~2.19, Proposition~2.11 (combined with Lemma~1.1), and Proposition~2.12~a).  More precisely, we choose $\cS^{\om}(0,n)$ as the largest cluster $\cC^{\vee}(Q_1)$ where $Q_1$ is the smallest special cube appearing in the proof of \cite[Theorem 2.18]{Ba04}.  In this case, our result on the quenched invariance principle Theorem \ref{thm:QIP:X} contains the ones in \cite{SS04, MP07, BP07, BB07}.
\end{example}
\begin{example}[Percolation clusters in models with long-range correlations]
  Consider a family of probability measure $\prob^{u}$ on $\{0, 1\}^{\bbZ^d}$ indexed by $u \in (a,b)$ that satisfies the assumptions \textbf{P1--P3}, \textbf{S1} and \textbf{S2} in \cite{Sa14}.  For a given sample $\{\eta(x) : x \in \bbZ^d\}$ of $\prob^u$, we set 
\begin{align*}
  \om(\{x,y\}) \;=\; \eta(x) \cdot \eta(y)
  \qquad \forall\, \{x,y\} \in E_d.
\end{align*}
For any fixed $u \in (a,b)$, set $\prob = \prob^u \circ\, \om^{-1}$.  Obviously, $\prob$ is ergodic with respect to translations of $\bbZ^d$.  In view of \cite[Remark~1.9 (2)]{Sa14}, there exists $\prob$-a.s.\ a unique infinite cluster.  Hence, Assumption~\ref{assumption:law} is satisfied.  Moreover, Assumption~\ref{assumption:cluster} on $\th$-very regular balls for any $\th \in (0, 1)$ follows from \cite[Proposition~4.3]{Sa14} with $\ve = 1/d$.  Therefore, the QFCLT for the simple random walk on percolation clusters given by $\om$ holds true.  In particular, the strategy used in showing Theorem~\ref{thm:QIP:X} provides an alternative proof of \cite[Theorem~1]{PRS13}.
\end{example}
Let us consider a more general model in which random walks move on percolation clusters with arbitrary jump rates.
\begin{example}[RCM defined by level sets of the Gaussian Free Field]
  Consider the discrete Gaussian Free Field $\phi = \{\phi(x) : x \in \bbZ^d\}$ for $d \geq 3$, i.e.\ $\phi$ is a Gaussian field with mean zero and covariances given by the Green function of the simple random walk on $\bbZ^d$.  The excursion set of the field $\phi$ above level $h$ is defined as $V_{\geq h}(\phi) \ldef \{x \in \bbZ^d : \phi(x) \geq h \}$, which can be considered as vertex set of the random graph of with edge set $E_{\geq h}(\phi) \ldef \{\{x,y\} : \phi(x) \wedge \phi(y) \geq h \}$.  It is well known \cite{BLM87, RS13} that there exists a threshold $h_* = h_*(d) \in [0, \infty)$ such that almost surely the graph $\big(V_{\geq h}(\phi), E_{\geq h}(\phi)\big)$ contains
  \begin{itemize}
  \item[(i)] for $h < h_*$, a unique infinite connected component;
  \item[(ii)] for $h > h_*$, only finite connected components.
  \end{itemize}
  We are interested in the first case, where the family $\{\prob^{h_*-h}$, $h \in (a,b)\},$ with $\prob^u$ denoting the law of the site percolation process $\{\indicator_{\phi(x)\geq u}:x\in\mathbb{Z}^d\}$, satisfies for some $0 < a < b <\infty$ the assumptions \textbf{P1--P3}, \textbf{S1} and \textbf{S2} in \cite{Sa14} (for more details, see Subsection~1.1.2 therein).  For $h \in (a,b)$, define 
  \begin{align*}
    \om(\{x, y\})
    &\;=\;
    \exp\!\big(\phi(x) + \phi(y) \big)\,
    \indicator_{|\phi(x)| \wedge |\phi(y)| \,\geq\, h_*-h}
    \qquad \forall\, \{x,y\} \in E_d,
  \end{align*}
  and denote by $\prob$ the corresponding law.  In view of \cite[Proposition 4.3]{Sa14}, Assumptions \ref{assumption:law} and \ref{assumption:cluster} are satisfied.  Since $\mean[\om(e)^p] < \infty$ and $\mean[\om(e)^{-q} \indicator_{\om(e)>0}] < \infty$ for every $p, q \in (0,\infty)$, Theorem~\ref{thm:QIP:X} holds for this random conductance model.
\end{example}

\subsection{The method}
We follow the most common approach to prove a QFCLT that is based on \emph{harmonic embedding}, see \cite{Bi11} for a detailed exposition of this method.  A key ingredient of this approach is the \emph{corrector}, a random function,  $\chi\!:\Om \times \bbZ^d \to \bbR^d$ satisfying $\bbP_0$-a.s.\ the following \emph{cocycle property}
\begin{align*}
  \chi(\om, x + y) - \chi(\om, x)
  \;=\;
  \chi(\tau_x\om, y),
  \qquad x,y \in \cC_{\infty}(\om).
\end{align*}
such that $|\chi(\om, x)| = o(|x|)$ as $|x| \to \infty$ and
\begin{align*}
  \Phi(\om, x) \;=\; x - \chi(\om, x)
\end{align*}
is an $\cL^\om$-harmonic function in the sense that $\prob_0$-a.s.\
\begin{align*}
  \cL^\om \Phi(\om, x)
  \;=\;
  \sum_{y} \om(\{x,y\}) \big( \Phi(\om, y) - \Phi(\om, x) \big)
  \;=\;
  0,
  \qquad \forall\, x \in \cC_{\infty}(\om).
\end{align*}
This can be rephrased by saying that $\chi$ is a solution of the Poisson equation
\begin{align*}
  \cL^{\om} u \;=\; \cL^{\om} \Pi
\end{align*}
where $\Pi$ denotes the identity mapping on $\bbZ^d$.  The existence of $\chi$ is guaranteed by Assumption~\ref{assumption:law}.  Further, the $\cL^{\om}$-harmonicity of $\Phi$ implies that
\begin{align*}
  M_t \;=\; X_t - \chi(\om, X_t)
\end{align*}
is a martingale under $\Prob_{\!0}^{\om}$ for $\prob_0$-a.e.\ $\om$, and a QFCLT for the martingale part $M$ can be easily shown by standard arguments.  In order to obtain a QFCLT for the process $X$, by Slutsky's theorem, it suffices to show that for any $T > 0$ and $\prob_0$-a.e\ $\om$
\begin{align}\label{eq:conv:prob:intro}
  \sup_{0 \,\leq\, t \,\leq\, T}\,
  \frac{1}{n}\, \big| \chi(\om, X_{t n^2}) \big|
  \;\underset{n \to \infty}{\longrightarrow}\;
  0 
  \quad \text{ in $\Prob_{0}^\om$-probability},
\end{align}
which can be deduced from $\ell^{\infty}$-sublinearity of the corrector:
\begin{align} \label{eq:sublin_intro}
  \lim_{n \to \infty}  \max_{x \in B^{\om}(0,n)} \frac{1}{n}\, \big| \chi(\om, x) \big|
  \;=\;
  0
  \qquad \prob_0\!\text{-a.s}.
\end{align}

The main challenge in the proof of the QFCLT is to show \eqref{eq:sublin_intro}.  In a first step we show that the rescaled corrector converges to zero $\prob_0$-a.s.\ in the space averaged norm $\|\!\cdot\!\|_{1, B^{\om}(0, n)}$ (see Proposition~\ref{prop:sublinearity:1} below).  A key ingredient in the proof is a new anchored relative isoperimetric inequality (Lemma~\ref{lemma:riso:weight}) and an extension of Birkhoff's ergodic theorem, see Appendix~\ref{appendix:ergodic} for more details.  In a second step, we establish a maximal inequality for the solution of a certain class of Poisson equations using a Moser iteration scheme.  As an application, the maximum of the rescaled corrector in the ball $B^{\om}(0, n)$ can be controlled by the corresponding $\|\!\cdot\!\|_{1, B^{\om}(0, n)}$-norm.  In the case of elliptic conductances such a Moser iteration has already been implemented in order to show a QFCLT \cite{ADS15}, a local limit theorem and elliptic and parabolic Harnack inequalities \cite{ADS16a} as well as upper Gaussian estimates on the heat kernel \cite{ADS16b}.  The Moser iteration is based on a Sobolev inequality for functions with compact support which follows in the case of elliptic conductances ($\cC_{\infty}(\om) \equiv \bbZ^d$) from the isoperimetric inequality \eqref{eq:iso:Zd} on $\bbZ^d$.  Since such an isoperimetric inequality on random graphs is true only for sufficiently large sets (Lemma~\ref{lemma:iso:large_sets:random}), the present proof of the Sobolev inequality relies on an interpolation argument in order to deal with the small sets (see Lemma~\ref{lemma:iso:all_sets} below). 
\medskip

The paper is organized as follows: In Section~\ref{sec:QIP}, we prove our main result.  After recalling the construction of the corrector and proving the convergence of the martingale part, we show the $\ell^1$- and $\ell^{\infty}$-sublinearity of the corrector.  The proof of the $\ell^1$-sublinearity is based on an anchored Sobolev inequality that we show in a more general context in Section~\ref{sec:Sobolev}.  Finally, the Appendix contains an ergodic theorem that is needed in the proofs.

Throughout the paper, we write $c$ to denote a positive constant that may change on each appearance, whereas constants denoted by $C_i$ will be the same through each argument. 

\section{Quenched invariance principle}\label{sec:QIP}
Throughout this section we suppose that Assumption~\ref{assumption:law} holds.

\subsection{Harmonic embedding and the corrector}
In this subsection, we first construct a corrector to the process $X$ such that $M_t = X_t - \chi(\om, X_t)$ is a martingale under $\Prob_{\!0}^{\om}$ for $\prob$ a.e.\ $\om$.  Second, we prove an invariance principle for the martingale part.
\begin{definition} 
  A measurable function, also called a random field, $\Psi\!: \Om \times \bbZ^d \to \bbR$ satisfies the cocycle property if for $\prob_0$-a.e. $\om$, it hold that
  \begin{align*}
    \Psi(\tau_x\om, y-x)
    \;=\;
    \Psi(\om, y) - \Psi(\om, x),
    \qquad \text{for } x, y \in \cC_{\infty}(\om).
  \end{align*}
  We denote by $L^2_\mathrm{cov}$ the set of functions $\Psi\!: \Om \times \bbZ^d \to \bbR$ satisfying the cocycle property such that
  \begin{align*}
    \Norm{\Psi}{L_\mathrm{cov}^2}^2
    \;\ldef\;
    \mean_0\!\Big[{\textstyle \sum_{x \sim 0}}\, \om(\{0, x\})\, |\Psi(\om,x)|^2\Big]
    \;<\;
    \infty.
  \end{align*}
\end{definition} 
In the following lemma we summerize some properties of functions in $L_{\mathrm{cov}}^2$.
\begin{lemma}\label{lemma:l2cov}
  For all $\Psi \in L^2_{\mathrm{cov}}$, we have
  \begin{enumerate}[(i)]
    \setlength{\itemsep}{.5ex}
  \item $\Psi(\om, 0) = 0$ and $\Psi(\tau_x \om, -x) = \Psi(\om, x)$ for any $x \in \cC_{\infty}(\om)$ and $\om \in \Om_0$,
  \item $\Norm{\Psi}{L^2_{\mathrm{cov}}} = 0$, if and only if, $\Psi(\om, x) = 0$ for all $x \in \cC_{\infty}(\om)$ and $\prob_0$-a.e.\ $\om \in \Om_0$.
  \end{enumerate}
\end{lemma}
\begin{proof}
  (i) follows from the definition.\\
  (ii) ''$\Leftarrow$''  The assertion follows immediately from the  definition of $\|\cdot\|_{L_{\mathrm{cov}}^2}$.\\
  ''$\Rightarrow$''  Suppose that $\Norm{\Psi}{L_{\mathrm{cov}}^2} = 0$.  By using the stationarity of $\prob$ and the cocycle property, we obtain that, for any $y \in \bbZ^d$,
  \begin{align}\label{eq:l2cov:shift}
    0
    &\;=\;
    \mean\!%
    \Big[
      {\textstyle \sum_{x \sim 0}}\,
      (\tau_y\om)(\{0,x\})\, \Psi(\tau_y \om, x)^2\,
      \indicator_{0 \in \cC_{\infty}(\tau_y \om)}
    \Big]
    \nonumber\\[.5ex]
    &\;=\;
    \mean\!%
    \Big[
      {\textstyle \sum_{x \sim 0}}\, \om(\{y, y+x\})\,
      \big|\Psi(\om, y + x) - \Psi(\om,y)\big|^2\,
      \indicator_{y \in \cC_{\infty}(\om)}
    \Big].
  \end{align}
  Hence, for any $y \in \bbZ^d$ there exists $\Om_y^* \subset \Om$ such that $\prob[\Om_y^*] = 1$ and for all $\om \in \Om_*$
  \begin{align}\label{eq:l2cov:zero}
    \om(\{y,y + x\})\, \big|\Psi(\om,y + x) - \Psi(\om,y)\big|^2\,
    \indicator_{y \in \cC_{\infty}(\om)}
    \;=\;
    0
    \quad \forall\; |x| = 1. 
  \end{align}
  Set $\Om^*\ldef \bigcap_{y \in \bbZ^d} \Om_y^*$.  Obviously, $\prob[\Om^*] = 1$, and for any $\om \in \Om^*$, \eqref{eq:l2cov:zero} holds true for all $y \in \bbZ^d$.  In particular, for any $\om \in \Om^* \cap \Om_0$ and $z \in \cC_{\infty}(\om)$, there exist $z_0 = 0, \ldots, z_k = z$ with $\{z_i, z_{i+1}\} \in \cO(\om)$ for all $0 \leq i \leq k-1$ such that
  \begin{align*}
    \Psi(\om, z_i) \;=\; \Psi(\om, z_{i+1})
    \qquad 
    \forall\, i = 0, \ldots, k-1.
  \end{align*}
  Hence, $\Psi(\om, z) = \Psi(\om, 0) = 0$.  This completes the proof.
\end{proof}
In particular, it can be checked that $L_{\mathrm{cov}}^2$ is a Hilbert space (cf. \cite{BP07,MP07}).

We say a function $\vp\!: \Om \to \bbR$ is \emph{local} if it only depends on the value of $\om$ at a finite number of edges.  We associate to $\vp$ a (horizontal) gradient $\mD \vp\!: \Om \times \bbZ^d \to \bbR$ defined by
\begin{align*}
  \mD \vp (\om,x)
  \;=\;
  \vp(\tau_x \om) - \vp(\om),
  \qquad x \in \bbZ^d.
\end{align*}
Obviously, if the function $\vp$ is bounded, $\mD \vp$ is an element of $L_{\mathrm{cov}}^2$.  Following \cite{MP07}, we introduce an orthogonal decomposition of the space $L_{\mathrm{cov}}^2$.  Set
\begin{align*}
  L_{\mathrm{pot}}^2
  \;=\;
  \mathop{\mathrm{cl}}
  \big\{ \mD \vp \mid \vp\!: \Om \to \bbR\; \text{ local} \big\}
  \;\text{ in }\;  L_{\mathrm{cov}}^2,
\end{align*}
being the closure in $L_{\mathrm{cov}}^2$ of the set gradients and let $L_{\mathrm{sol}}^2$ be the orthogonal complement of $L_{\mathrm{pot}}^2$ in $L_{\mathrm{cov}}^2$, that is
\begin{align*}
  L_{\mathrm{cov}}^{2}
  \;=\;
  L_{\mathrm{pot}}^2 \oplus L_{\mathrm{sol}}^2.
\end{align*}
In order to define the corrector, we introduce the \emph{position field} $\Pi\!: \Om \times \bbZ^d \to \bbR^d$ with $\Pi(\om, x) = x$.  We  write $\Pi_j$ for the $j$-th coordinate of $\Pi$.  Since $ \Pi_j(\tau_x\om, y-x) = \Pi_j(\om, y) - \Pi_j(\om, x)$ for all $x, y \in \bbZ^d$, the $j$-th component of the position field $\Pi_j$ satisfies the cocycle property for every $\om \in \Om_0$.  Moreover,
\begin{align}
  \Norm{\Pi_j}{L_\mathrm{cov}^2}^2 
  \;=\;
  \mean_0\!\Big[{\textstyle \sum_{x \sim 0}}\; \om(\{0,x\})\, |x_j|^2\Big]
  \;=\;
  2 \mean_0\!\big[\om(\{0, e_j\})\big]
  \;<\;
  \infty, 
\end{align}
where $e_j$ denotes the $j$-th coordinate unit vector.  Hence, $\Pi_j \in L_{\mathrm{cov}}^2$.  So, we can define $\chi_j \in L_{\mathrm{pot}}^2$ and $\Phi_j \in L_{\mathrm{sol}}^2$ as follows
\begin{align*}
  \Pi_j
  \;=\;
  \chi_j \,+\, \Phi_j
  \;\in\;
  L_{\mathrm{pot}}^2 \oplus L_{\mathrm{sol}}^2.
\end{align*}
This defines the corrector $\chi = (\chi_1, \dots, \chi_d) : \Om \times \bbZ^d \to \bbR^d$.  Further, we set
\begin{align}\label{eq:def:M}
  M_t \;=\; \Phi(\om, X_t) \;=\; X_t - \chi(\om, X_t).
\end{align}
The following proposition summarizes the properties of $\chi$, $\Phi$ and $M$; see, for example, \cite{ABDH13}, \cite{BD10} or \cite{Bi11} for detailed proofs.
\begin{prop} \label{prop:constr_corr}
  For $\prob_0$-a.e.\ $\om$, we have
  \begin{align} \label{eq:Phi}
    \cL^{\om} \Phi(x)
    \;=\;
    \sum_{y \sim x} \om(\{x,y\}) \big(\Phi(\om, y) - \Phi(\om,x) \big)
    \;=\;
    0
    \qquad \forall\, x \in \cC_{\infty}(\om).
  \end{align}
  In particular, for $\prob_0$-a.e.\ $\om$ and for every $v \in \bbR^d$, $M$ and $v\cdot M$ are $\Prob_0^{\om}$-martingales with respect to the filtration $\cF_t = \si(X_s, s \leq t)$.  The quadratic variation process of the latter is given by
  \begin{align}
    \langle v \cdot M \rangle_t
    \;=\;
    \int_0^t
      \sum_{x \sim 0}\, (\tau_{X_s}\om )(\{0,x\})\,
      \big(v \cdot \Phi(\tau_{X_s}\om,x)\big)^2\;
    \md s.
  \end{align}
\end{prop}
In the sequel, we prove a quenched invariance principle for the martingale part.  This is standard and follows from the ergodicity of the process of the \emph{environment as seen from the particle} $\{\tau_{X_t} \om : t \geq 0\}$ which is a Markov process taking values in the environment space $\Om_0$ with generator
\begin{align*}
  \widehat{\cL}\, \vp(\om)
  \;=\;
  \sum_{x \sim 0} \om(\{0,x\})\, \big( \vp(\tau_x\om) - \vp(\om) \big)
\end{align*}
acting on bounded functions $\vp\!:\Om_0 \to \bbR$.  The following result is a generalization of Kozlov's theorem \cite{Ko85} in the case that the underlying random walk is reversible.
\begin{lemma}\label{lemma:ENVI}
  The measure $\prob_0$ is reversible, invariant and ergodic for the environment process $\{\tau_{X_t} \om : t \geq 0\}$.
\end{lemma}
\begin{proof}
  The reversibility of $\{\tau_{X_t} \om : t \geq 0\}$ with respect to $\prob_0$ follows directly from Assumption~\ref{assumption:law}.   The proof of the ergodicity of the environmental process relies on the ergodicity of $\prob$ with respect to shifts of $\bbZ^d$ and the fact that for $\prob$-a.e.\ $\om$ the infinite cluster, $\cC_{\infty}(\om)$, is unique.  See \cite[Lemma~4.9]{dMFGW89} for a detailed proof.
\end{proof}
In the next proposition we show both the convergence of the martingale part and the non-degeneracy of the limiting covariance matrix.  The proof of the latter, inspired by the argument given in \cite{PRS13} (see also \cite{BB07}), relies on the $\ell^1$-sublinearity of the corrector that we will show below in Proposition~\ref{prop:sublinearity:1}.  
\begin{prop}[QFCLT for the martingale part]\label{prop:QIP:martingale}
  For $\prob_0$-a.e.\ $\om$, the sequence of processes $\{\tfrac{1}{n} M_{t n^2} : t \geq 0\}$ converges in $\Prob_0^{\om}$-probability to a Brownian motion with a deterministic covariance matrix $\Si^2$ given by
  \begin{align*}
  \Si_{ij}^2
  \;=\;
  \mean_0\!%
  \Big[
    {\textstyle \sum_{x \sim 0}}\; \om(\{0,x\})\, \Phi_i(\om, x)\, \Phi_j(\om, x)
  \Big].
  \end{align*}
  Additionally, if $\th \in (0, 1)$ satisfies Assumption~\ref{assumption:cluster} and $\mean[(1/\om(e)) \indicator_{e \in \cO}] < \infty$ for any $e \in E_d$, then the limiting covariance matrix $\Si^2$ is non-degenerate.
\end{prop}
\begin{proof}
  The proof follows from the martingale convergence theorem by Helland, cf.\ \cite[Theorem~5.1a)]{He82}; see also \cite{ABDH13} or \cite{MP07} for details. The argument is based on the fact that the quadratic variation of $\{\tfrac{1}{n} M_{t n^2} : t \geq 0\}$ converges, for which the ergodicity of the environment process in Lemma~\ref{lemma:ENVI} is needed.

  It remain to show that the limiting Brownian motion is non-degenerate.  The argument is similar to the one in \cite{PRS13}, but avoids the use of the $\ell^{\infty}$-sublinearity of the corrector.  Assume that $(v, \Si^2 v) = 0$ for some $v \in \bbR^d$ with $|v| = 1$.  First, we deduce from Lemma~\ref{lemma:l2cov} that, for $\prob_0$-a.e.\ $\om$, $v \cdot \Phi(\om, x) = 0$ for all $x \in \cC_{\infty}(\om)$.  Since $x = \chi(\om, x) + \Phi(\om, x)$, this implies that, for $\prob_0$-a.e. $\om$, $|v \cdot x| = |v \cdot \chi(\om, x)|$ for all $x \in \cC_{\infty}(\om)$.  In particular,
  \begin{align}\label{eq:nondeg:1}
    \frac{1}{|B^{\om}(n)|}
    \sum_{x \in B^{\om}(n)}\mspace{-6mu}
    \big|v \cdot \tfrac{1}{n} x\big|
    \;=\;
    \frac{1}{|B^{\om}(n)|}
    \sum_{x \in B^{\om}(n)} \mspace{-6mu}
    \big|v \cdot \tfrac{1}{n} \chi(\om, x)\big|.
  \end{align}
  In view of Proposition~\ref{prop:sublinearity:1}, the right-hand side of \eqref{eq:nondeg:1} vanishes for $\prob_0$-a.e.\ $\om$ as $n$ tends to infinity.  On the other hand, for any $\de \in (0, 1)$ we have that
  \begin{align*}
    \frac{1}{n^{d}}\! \sum_{x \in B^{\om}(n)} \mspace{-6mu}
    \big|v \cdot \tfrac{1}{n} x\big|
    &\;\geq\;
    \frac{\de^2}{n^d}\! \sum_{\substack{x \in B^{\om}(n)\\x \ne 0}} \mspace{-6mu}
    \indicator_{|x| > \de n}\; \indicator_{|v \cdot x/|x|| > \de}
    \\[.5ex]
    &\;\geq\;
    \frac{\de^2}{n^d}
    \bigg(
      |B^{\om}(n)| \,-\, |B(\de n)|
      \,- \sum_{\substack{x \in B(n)\\ x \ne 0}} \mspace{-4mu}
      \indicator_{|v \cdot x/|x|| \leq \de}
    \bigg)
  \end{align*}
  Due to \eqref{eq:ass:balls}, $|B^{\om}(n)| \geq C_{\mathrm{V}} n^d$ for all $n \geq N_1(\om)$ and $\prob_0$-a.e.\ $\om$.  Moreover, the other two terms in the bracket above are of order $\de n^d$.  Hence, by choosing $\de$ sufficiently small, there exists $c > 0$ such that 
  \begin{align*}
    \liminf_{n \to \infty}\,
    \frac{1}{|B^{\om}(n)|}
    \sum_{x \in B^{\om}(n)}\mspace{-6mu}
    \big|v \cdot \tfrac{1}{n} x\big|
    \;\geq\;
    c
    \;>\;
    0.
  \end{align*}
  Thus, we proved that $(v, \Si^2 v) > 0$ for all $0 \ne v \in \bbR^d$, which completes the proof.
\end{proof}

\subsection{Sublinearity of the corrector}
Recall that we denote by $B^{\om}(x, r)$ and $B(x, r)$ a closed ball with center $x \in \cC_{\infty}(\om)$ and radius $r \geq 0$ with respect to the graph distance $d^{\om}$ and usual $\ell^1$-distance on $\bbZ^d$, respectively.  To lighten notation, we write $B^{\om}(r) \equiv B^{\om}(0, r)$ and $B(r) \equiv B(0, r)$.  Further, for any non-empty $A \subset \bbZ^d$, we define a locally space-averaged norm for functions $f\!: \bbZ^d \to \bbR$ by
\begin{align*}
  \Norm{f}{p, A}
  \;\ldef\;
  \bigg(
    \frac{1}{|A|}\, \sum_{x \in A} |f(x)|^p
  \bigg)^{\!\!1/p},
  \qquad p \in [1, \infty).
\end{align*}
Our main objective in this subsection is to prove the $\ell^{\infty}$-sublinearity of the corrector.
\begin{prop}[$\ell^{\infty}$-sublinearity]\label{prop:sublinearity:infty}
  Suppose that $\th \in (0, 1)$ and $p, q \in [1, \infty]$ satisfy Assumptions~\ref{assumption:cluster} and \ref{assumption:pq}.  Then, for any $j = 1, \ldots, d$,
  \begin{align}
    \lim_{n \to \infty}\, \max_{x \in B^{\om}(n)}\,
    \big|\tfrac{1}{n}\chi_j(\om, x)\big|
    \;=\;
    0,
    \qquad
    \prob_0\text{-a.s.}
  \end{align}
\end{prop}
The proof is based on both ergodic theory and purely analytic tools.  In a first step, we show the $\ell^1$-sublinearity of the corrector, that is the convergence of $\frac{1}{n} \chi$ to zero in the $\Norm{\,\cdot\,}{1, B^{\om}(n)}$-norm.  This proof uses the spatial ergodic theorem as well as the anchored $S_1$-Sobolev inequality that we established in Proposition~\ref{prop:Sobolev1:weighted}.  In a second step, we use the maximum inequality in order to bound from above the maximum of $\frac{1}{n} \chi$ in $B^{\om}(n)$ by $\Norm{\frac{1}{n} \chi}{1, B^{\om}(n)}$.

Let us start with some consequences from the ergodic theorem.  To simplify notation let us define the following measures $\mu^{\om}$ and $\nu^{\om}$ on $\bbZ^d$
\begin{align*}
  \mu^{\om}(x)
  \;=\;
  \sum_{x \sim y} \om(\{x,y\})
  \qquad \text{and} \qquad
  \nu^{\om}(x)
  \;=\;
  \sum_{x \sim y} \frac{1}{\om(\{x, y\})}\, \indicator_{\{x,y\} \in \cO(\om)},
\end{align*}
where we still use the convention that $0/0 = 0$.
\begin{lemma}\label{lemma:ergodic:weights}
  Suppose that for $\prob_0$-a.e.\ $\om$ there exists $N_1(\om) < \infty$ such that the ball $B^{\om}(n)$ satisfies the volume regularity \eqref{eq:ass:balls} for all $n \geq N_1(\om)$.  Further, assume that $\mean[\om(e)^p] < \infty$ and $\mean[(1/\om(e))^q \indicator_{e \in \cO}] < \infty$ for some $p, q \in [1, \infty)$.  Then, for $\prob_0$-a.s.\ $\om$ there exists $c < \infty$ such that
  \begin{align}
    \limsup_{n \to \infty} \Norm{\mu^{\om}}{p, B^{\om}(n)}^p
    \;\leq\;
    c\, \mean_0[\mu^{\om}(0)^p]
    \quad \text{and} \quad
    \limsup_{n \to \infty} \Norm{\nu^{\om}}{q, B^{\om}(n)}^q
    \;\leq\;
    c\, \mean_0[\nu^{\om}(0)^q].
  \end{align}
\end{lemma}
\begin{proof}
  The assertions follows immediately from the spatial ergodic theorem.  For instance, we have for $\prob_0$-a.s.
  \begin{align*}
    \limsup_{n \to \infty} \Norm{\mu^{\om}}{p, B^{\om}(n)}^p
    \overset{\eqref{eq:ass:balls}}{\;\leq\;}
    \limsup_{n \to \infty}
    \frac{C_{\mathrm{V}}^{-1}}{n^d}\,
    \sum_{x \in B(n)} \mspace{-3mu}\mu^{\tau_{x} \om}(0)^p\,
    \indicator_{0 \in \cC(\tau_x \om)}
    \;\leq\;
    c\, \mean_0\!\big[\mu^{\om}(0)^p\big],
  \end{align*}
  where we exploit the observation that $B^{\om}(n) \subset B(n) \cap \cC_{\infty}(\om)$ for every $n \geq 1$.
\end{proof}
The next lemma relies on an extension of Birkhoff's ergodic theorem that we show in the appendix.
\begin{lemma}\label{lemma:weighted:grad}
  Let $w_n\!: E_d \to (0, \infty)$ be defined by $w_n(\{x,y\}) = (n/\max\{|x|_1, |y|_1\})^{d-\ve}$ for some $\ve \in (0, 1)$, and assume that $\mean[(1/\om(e)) \indicator_{e \in \cO}] < \infty$ for all $e \in E_d$.  Then, there exists $C_5 < \infty$ such that for any $\Psi \in L_{\mathrm{cov}}^2$ and $\prob_0$-a.e.\ $\om$,
  \begin{align}\label{eq:est:weighted:grad}
    \limsup_{n \to \infty} \frac{1}{n^d}\!
    \sum_{\substack{x,y \in B^{\om}(n)\\x \sim y}}\mspace{-6mu} w_n(\{x,y\})\,
    |\Psi(\om, x) - \Psi(\om, y)|
    \;\leq\;
    \frac{C_5}{\ve}\, \mean_0[\nu^{\om}(0)]^{1/2}\,
    \Norm{\Psi}{L_{\mathrm{cov}}^2}.
  \end{align}
\end{lemma}
\begin{proof}
  First, an application of the Cauchy-Schwarz inequality yields
  \begin{align}\label{eq:cocyle:CS}
    &\mean_0\!\Big[
      {\textstyle \sum_{0 \sim y}} |\Psi(\om, y)|\, 
      \indicator_{\{0, y\} \in \cO}
    \Big]
    \nonumber\\[.5ex]
    &\mspace{36mu}\leq\;
    \mean_0\!\Big[
      {\textstyle \sum_{0 \sim y}} \big(1/\om(\{0,y\})\big)\,
      \indicator_{\{0, y\} \in \cO}
    \Big]^{1/2}\,
    \mean_0\!\Big[
      {\textstyle \sum_{0 \sim y}} \om(\{0,y\})\, |\Psi(\om, y)|^2
    \Big]^{1/2}
    \nonumber\\[.5ex]
    &\mspace{36mu}=\;
    \mean_0[\nu^{\om}(0)]^{1/2}\, \Norm{\Psi}{L_{\mathrm{cov}}^2}.
  \end{align}
  which is finite since $\Psi \in L_{\mathrm{cov}}^2$ and $\mean[(1/\om(e)) \indicator_{e \in \cO}] < \infty$ by assumption.  Recall that $\Psi$ satisfies the cocycle property, that is $\Psi(\om, x) - \Psi(\om, y) = \Psi(\tau_{x} \om, y-x)$ for $\prob_0$-a.e.\ $\om$ and for every $x, y \in \cC_{\infty}(\om)$.  Since $B^{\om}(n) \subset B(n) \cap \cC_{\infty}(\om)$ for every $n \geq 1$, we obtain that, for any $\om \in \Om_0$,
  \begin{align*}
    &\frac{1}{n^d}\!
    \sum_{\substack{x,y \in B^{\om}(n)\\x \sim y}}\mspace{-6mu} w_n(\{x,y\})\,
    \big|\Psi(\tau_x\om, y - x)\big|
    \nonumber\\[.5ex]
    &\mspace{36mu}\leq\;
    \frac{1}{n^d}\!
    \sum_{x,y \in B(n)} \mspace{-6mu} w_n(\{x,y\})\,
    \big|\Psi(\tau_x \om, y-x)\big|\,
    \indicator_{0 \in \cC_{\infty}(\tau_x \om)}
    \indicator_{\{0, y-x\} \in \cO(\tau_x \om)}
    \nonumber\\[.5ex]
    &\mspace{36mu}\leq\;
    \frac{\psi(\om)}{n^{\ve}}
    \,+\,
    \frac{1}{n^{d}}
    \sum_{\substack{x \in B(n)\\ x \ne 0}} \frac{\psi(\tau_x \om)}{|x/n|^{d-\ve}}
  \end{align*}
  where we introduced $\psi(\om) = \sum_{y \sim 0}|\Psi(\om, y)| \indicator_{0 \in \cC_{\infty}(\om)} \indicator_{\{0, y\} \in \cO(\om)}$ to lighten notation.  Further, an application of \eqref{eq:ergodic:conv} yields 
  \begin{align*}
    &\limsup_{n \to \infty}
    \frac{1}{n^d}\!
    \sum_{\substack{x,y \in B^{\om}(n)\\x \sim y}}\mspace{-6mu} w_n(\{x,y\})\,
    \big|\Psi(\tau_x\om, y - x)\big|
    \\[.5ex]
    &\mspace{36mu}\leq\;
    \frac{C_5}{\ve}\,
    \mean\!\Big[
      {\textstyle \sum_{0 \sim y}} |\Psi(\om, y)|\, \indicator_{0 \in \cC_{\infty}}
      \indicator_{\{0, y\} \in \cO}
    \Big]
    \overset{\eqref{eq:cocyle:CS}}{\;\leq\;}
    \frac{C_5}{\ve}\, \mean_0[\nu^{\om}(0)]^{1/2}\, \Norm{\Psi}{L_{\mathrm{cov}}^2},
  \end{align*}
  which concludes the proof.
\end{proof}
\begin{prop}[$\ell^{1}$-sublinearity]\label{prop:sublinearity:1}
  Suppose $\th \in (0, 1)$ satisfies Assumption~\ref{assumption:cluster}, and assume that $\mean\big[(1/\om(e)) \indicator_{e \in \cO}\big] < \infty$ for all $e \in E_d$.  Then, for any $j = 1, \ldots, d$,
  \begin{align}\label{eq:sublinearity:1}
    \lim_{n \to \infty}\, \frac{1}{|B^{\om}(n)|}
    \sum_{x \in B^{\om}(n)}\mspace{-6mu} \big| \tfrac{1}{n}\, \chi_j(\om, x)\big|
    \;=\;
    0,
    \qquad \prob\text{-a.s.}
  \end{align}
\end{prop}
\begin{proof}
  Since $\chi_j \in L_{\mathrm{pot}}^2$, there exists a sequence of bounded functions $\vp_{j,k}\!:\Om \to \bbR$ such that $\mathrm{D}\vp_{j,k} \to \chi_j$ in $L_{\mathrm{cov}}^{2}$ as $k \to \infty$.  Thus, for any fixed $k \geq 1$ we obtain
  \begin{align}\label{eq:sublinear:1:est1}
    \frac{1}{n^{d+1}}\!\sum_{x \in B^{\om}(n)}\mspace{-4mu} |\chi_j(\om, x)|
    \;\leq\;
    \frac{c \Norm{\vp_{j,k}}{L^{\!\infty}(\Om)}}{n}
    \,+\,
    \frac{1}{n^{d+1}}\!\sum_{x \in B^{\om}(n)}\mspace{-4mu}
    \big|(\chi_j - \mathrm{D} \vp_{j,k})(\om, x)\big|.
  \end{align}
  In order to bound from above the second term on the right hand-side of \eqref{eq:sublinear:1:est1} we consider the deterministic edge weight $w_n\!: E_d \to (0, \infty)$ that is defined by $w_n(\{x,y\}) = (n/\max\{|x|_1, |y|_1\})^{d-\ve}$ for some $\ve \in (0, 1)$.  Since $d^{\om}(x,y) \geq |x-y|_1$ for any $x,y \in \cC_{\infty}(\om)$, the $w_n$ satisfies the assumption in Proposition~\ref{prop:Sobolev1:weighted}.  By applying \eqref{eq:Sobolev1:weighted} and the cocycle property, we find for any $\om \in \Om_0$ that
  \begin{align*}
    &\frac{1}{n^{d+1}}\!
    \sum_{x \in B^{\om}(n)}\mspace{-6mu}
    \big|(\chi_j - \mathrm{D} \vp_{j,k})(\om,x)\big|
    \nonumber\\[.5ex]
    &\mspace{36mu}\leq\;
    \frac{\overbar{C}_{\mathrm{S}_1}}{n^d}\!
    \sum_{x, y \in B^{\om}(C_{\mathrm{W}} n)}\mspace{-18mu}
    w_n(\{x,y\})\,
    \big|(\chi_j - \mathrm{D}\vp_{j,k})(\tau_x \om, y-x)\big|\,
    \indicator_{\{0,y-x\} \in \cO(\tau_x \om)}.
  \end{align*}
  Hence, by combining the estimate above with \eqref{eq:sublinear:1:est1}, we get
  \begin{align}
    &\frac{1}{n^{d+1}}\!\sum_{x \in B^{\om}(n)}\mspace{-4mu} |\chi(\om, x)|
    \nonumber\\[.5ex]
    &\mspace{36mu}\leq\;
    \frac{c \Norm{\vp_{j,k}}{L^{\!\infty}(\Om)}}{n}
    \,+\,
    \frac{\overbar{C}_{\mathrm{S}_1}}{n^d}\mspace{-6mu}
    \sum_{\substack{x, y \in B^{\om}(C_{\mathrm{W}} n)\\\{0,y-x\} \in \cO(\tau_x \om)}}%
    \mspace{-18mu}
    w_n(\{x,y\})\,
    \big|(\chi_j - \mathrm{D}\vp_{j,k})(\tau_x \om, y-x)\big|.
  \end{align}
  In view of Lemma~\ref{lemma:weighted:grad}, we obtain that there exists $c < \infty$ such that for $\prob_0$-a.e.\ $\om$,
  \begin{align*}
    \lim_{k \to \infty}
    \lim_{n \to \infty} \Norm{\tfrac{1}{n}\, \chi_j(\om, \cdot)}{1, B^{\om}(n)}
    &\overset{\eqref{eq:ass:balls}}{\;\leq\;}
    \lim_{k \to \infty}
    \lim_{n \to \infty}\,
    \frac{C_{\mathrm{V}}^{-1}}{n^{d+1}}\!
    \sum_{x \in B^{\om}(n)}\mspace{-4mu} |\chi(\om, x)|
    \\[.5ex]
    &\overset{\eqref{eq:est:weighted:grad}}{\;\leq\;}
    \lim_{k \to \infty}
    \frac{c}{\ve}\, \mean_0[\nu^{\om}(0)]^{1/2}\,
    \Norm{\chi_j - \mathrm{D}\vp_{j,k}}{L_{\mathrm{cov}}^2}
    \;=\;
    0,
  \end{align*}
  which concludes the proof.
\end{proof}
In the following lemma we show that under the assumption that the ball $B^{\om}(n)$ is $\th$-very regular, the random graph $(\cC_{\infty}(\om), \cO(\om))$ satisfies $\prob_0$-a.s.\ an isoperimetric inequality for large sets.
\begin{lemma}[isoperimetric inequality for large sets]\label{lemma:iso:large_sets:random}
  Suppose that $\th \in (0, 1)$ satisfies Assumption~\ref{assumption:cluster}.  Then, for any $\om \in \Om_0$ and $n \geq N_0(\om)$, there exists $C_{\mathrm{iso}} \in (0, \infty)$ such that 
  \begin{align}\label{eq:iso:large_sets:random}
    |\partial^{\om} A| \;\geq\; C_{\mathrm{iso}}\, |A|^{(d-1)/d}
  \end{align}
  for all $A \subset B^{\om}(n)$ with $|A| \geq n^{\th}$.
\end{lemma}  
\begin{proof}
  First, note that \eqref{eq:iso:large_sets:random} follows trivially from \eqref{eq:ass:riso} for sets with $|A| \geq c n^d$.

  Consider $A \subset B^{\om}(n)$ with $|A| \geq n^{\th}$ and set $r^d \ldef (2 / C_{\mathrm{V}})\, |A|$.  Since $r \geq n^{\th/d}$, the Assumption~\ref{assumption:cluster} implies that any ball $B^{\om}(y, 3r)$ with $y \in B^{\om}(n)$ is regular.  Further, there exists a finite sequence $\{y_i \in B^{\om}(n) : i \in I\}$ such that
  \begin{align*}
    B^{\om}(y_i, r) \,\cap B^{\om}(y_j, r) \;=\; \emptyset
    \qquad \forall\, i, j \in I \text{  with } i \ne j
  \end{align*}
  and $B^{\om}(x, r) \cap \bigcup_{i \in I} B^{\om}(y_i, r) \ne \emptyset$ for all $x \in B^{\om}(n) \setminus \bigcup_{i \in I} B^{\om}(y_i, r)$.  Clearly, the sets $B^{\om}(y_i, 3r)$ cover the ball $B^{\om}(n)$, that is, $B^{\om}(n) \subset \bigcup_{i \in I} B(y_i, 3r)$. We claim that there exists $M < \infty$, independent of $n$, such that every $x \in B^{\om}(n)$ is contained in at most $M$ different balls $B^{\om}(y_i, 3C_{\mathrm{W}} r)$.  To prove this claim, set
  \begin{align*}
    I_x \;\ldef\; \big\{i \in I : x \in B^{\om}(y_i, 3C_{\mathrm{W}} r)\big\}.
  \end{align*}
  Notice that for any $i \in I_x$ we have that $B^{\om}(y_i, r) \subset B^{\om}(x, 4C_{\mathrm{W}}r)$.  By the fact that the sets $B^{\om}(y_i, r)$ are disjoint and regular, \eqref{eq:ass:balls} hold.  In particular,
  \begin{align*}
    C_0 (4 C_{\mathrm{W}})^d r^d
    \;\geq\;
    |B^{\om}(x, 4C_{\mathrm{W}} r)|
    \;\geq\;
    \sum_{x \in I_x} |B^{\om}(y_i, r)|
    \;\geq\;
    |I_x|\, C_{\mathrm{V}} r^d,
  \end{align*}
  where $C_0 \ldef \max_{k \geq 1} |B(k)|/k^d < \infty$.  Hence, $M \leq (4 C_{\mathrm{W}})^d c_0 / C_{\mathrm{V}}$ which completes the proof of the claim.  Further, set $A_i \ldef A \cap \cS(y_i, 3r)$.  Since $B^{\om}(y_i, 3r)$ is regular and $|A_i| \leq |A| \leq \frac{1}{2} |\cS(y_i,3r)|$ for any $i \in I$, \eqref{eq:ass:riso} implies that
  \begin{align*}
    |\partial^{\om} A|
    \;\geq\;
    \frac{1}{M}\, \sum_{i} |\partial_{\cS(y_i, r)}^{\om} A_i|
    \;\geq\;
    \frac{C_{\mathrm{riso}}}{M r} \sum_{i} |A_i|
    \;\geq\;
    \frac{C_{\mathrm{riso}}}{M r}\, |A|
    \;\geq\;
    \frac{C_{\mathrm{riso}} C_{V}}{2 M}\, |A|.
  \end{align*}
  By setting $C_{\mathrm{iso}} \ldef C_{\mathrm{riso}} C_{\mathrm{V}} / (2 M)$, the assertion \eqref{eq:iso:large_sets:random} follows.
\end{proof}
The next proposition relies on the application of the Moser iteration scheme that has been established for general graphs in \cite{ADS15}.  A key ingredient in this approach is the following Sobolev inequality 
\begin{align}
  \Norm{u}{d'/(d'-1), B^{\om}(n)}
  \;\leq\;
  C_{\mathrm{S}_1}\, \frac{n}{|B^{\om}(n)|}\,
  \sum_{x \vee y \in B^{\om}(n)}\mspace{-9mu}
  |u(x) - u(y)|\, \indicator_{\{x,y\} \in \cO(\om)}
\end{align}
for some suitable $d'$ that we will prove in Proposition~\ref{prop:sobolev:S1}.
\begin{prop}[maximal inequality]\label{prop:maximal:inequality}
  Suppose $\th \in (0, 1)$ and $p, q \in [1, \infty]$ satisfies Assumption~\ref{assumption:cluster} and \ref{assumption:pq}.  Then, for every $\al > 0$ there exist $\ga' > 0$ and $\ka' > 0$ and $c(p, q, \th, d) < \infty$ such that for any $\om \in \Om_0$ and $j = 1, \ldots, d$,
  \begin{align*}
    \max_{x \in B^{\om}(n)}\,
    \big|\tfrac{1}{n}\, \chi_j(\om, x)\big|
    \;\leq\;
    c\,
    \Big(
      1 \vee \Norm{\mu^{\om}}{p, B^{\om}(n)}\, \Norm{\nu^{\om}}{q, B^{\om}(n)}
    \Big)^{\!\ka'}\,
    \Norm{\tfrac{1}{n}\,\chi_j(\om, \cdot)}{\al, B^{\om}(2n)}^{\ga'}.
  \end{align*}
\end{prop}
\begin{proof}
  In view of Lemma~\ref{lemma:iso:large_sets:random} and Assumption~\ref{assumption:cluster}, for any $\om \in \Om_0$ and $n \geq N_0(\om)$ the assumptions of Proposition~\ref{prop:sobolev:S1} are satisfied.  Further, let $\ze = (1-\th)/(1-\th/d)$ and set $d' = (d-\th)/(1-\th)$.  Then, Proposition~\ref{prop:sobolev:S1} implies that 
  \begin{align*}
    \Norm{u}{d'/(d'-1), B^{\om}(n)}
    \;=\;
    \Norm{u}{d/(d-\ze), B^{\om}(n)}
    \;\leq\;
    C_{\mathrm{S}_1}\, \frac{n}{|B^{\om}(n)|}\,
    \sum_{\substack{x \vee y \in B^{\om}(n)\\\{x,y\} \in \cO(\om)}}\mspace{-9mu}
    |u(x) - u(y)|
  \end{align*}
  for any $u\!: \cC_{\infty}(\om) \to \bbR$ with $\supp u \subset B^{\om}(n)$.  By taking this inequality as a starting point and using the fact that by definition $\chi(\om,x) = 0$ for any $x \in \bbZ^d \setminus \cC_{\infty}(\om)$, the assertion for $\tfrac{1}{n} \chi_j(\om, \cdot)$ follows directly from \cite[Corollary~3.9]{ADS15} with $f(x) = \frac{1}{n} x_j$, $x_0 = 0$, $\si = 1$, $\si' = 1/2$, $n$ replaced by $2n$ and $d$ replaced by $d'$.
\end{proof}
Proposition~\ref{prop:sublinearity:infty} follows immediately from Proposition~\ref{prop:maximal:inequality} with the choice $\al=1$, combined with Proposition~\ref{prop:sublinearity:1} and Lemma~\ref{lemma:ergodic:weights}.  

\begin{proof}[Proof of Theorem~\ref{thm:QIP:X}]
  Proceeding as in the proof of \cite[Proposition~2.13]{ADS15} (with the minor modification that the exit time $T_{L,n}$ of the rescaled process $X^{(n)}$ from the cube $[-L, L]^d$ is replaced by $T_{L,n}^{\om} \ldef \inf\{t \geq 0 : X_{t n^2} \not\in B^{\om}(n)\}$), the {$\ell^{\infty}$-sublinearity} of the corrector that we have established in Proposition~\ref{prop:sublinearity:infty} implies that for any $T > 0$ and $\prob_0$-a.e.\ $\om$
  \begin{align*}
    \sup_{t \leq T} \big| \tfrac{1}{n}\, \chi(\om, X_{t n^2})|
    \;\underset{n \to \infty}{\longrightarrow}\;
    0,
    \qquad \text{in } \Prob_0^{\om}\text{-probability.}
  \end{align*}
  Thus, the assertion of Theorem~\ref{thm:QIP:X} now follows from Proposition \ref{prop:QIP:martingale}.
\end{proof}

\section{Sobolev inequalities on graphs}
\label{sec:Sobolev}
As seen in the previous section, both the Sobolev and the anchored Sobolev inequality turned out to be a crucial tool in order to prove the $\ell^1$- and $\ell^{\infty}$-sublinearity of the corrector.  In this section we will prove these inequalities for general graphs.

\subsection{Setup and preliminaries}
Let us consider an infinite, connected, locally finite graph $G = (V, E)$ with vertex set $V$ and edge set $E$.  Let $d$ be the natural graph distance on $G$.  We denote by $B(x,r)$ the closed ball with center $x$ and radius $r$, i.e. $B(x,r) \ldef \{y \in V \mid d(x,y) \leq \lfloor r \rfloor\}$.  The graph is endowed with the counting measure, i.e.\ the measure of $A \subset V$ is simply the number $|A|$ of elements in $A$.  Given a non-empty subset $B \subseteq V$, we define for any $A \subset B$ the \emph{relative boundary} of $A$ with respect to $B$ by
\begin{align*}
  \partial_B A
  \;\ldef\;
  \big\{
    \{x, y\} \in E \;:\;
    x \in A \text{ and } y \in B \setminus A
  \big\},
\end{align*}
and we simply write $\partial A$ instead of $\partial_V A$. We impose the following assumption on the properties of the graph $G$.
\begin{assumption}\label{ass:graph} 
  For some $d \geq 2$, there exist constants $c_{\mathrm{reg}}, C_{\mathrm{reg}}, C_{\mathrm{riso}}, C_{\mathrm{iso}} \in (0, \infty)$, $C_{\mathrm{W}} \in [1, \infty)$ and $\th \in (0, 1)$ such that for all $x \in V$ it holds
  \begin{itemize}
  \item[(i)] volume regularity of order $d$ for large balls: there exists $N_1(x) < \infty$ such that for all $n \geq N_1(x)$,
    \begin{align}\label{eq:reg:large_balls}
      c_{\mathrm{reg}}\, n^d \;\leq\; |B(x, n)| \;\leq\; C_{\mathrm{reg}}\, n^d.
    \end{align}
  \item[(ii)] (weak) relative isoperimetric inequality: there exists $N_2(x) < \infty$ and an increasing sequence $\{\cS(x, n) \subset V : n \in \bbN\}$ of connected sets such that for all $n \geq N_2(x)$,
    \begin{align}\label{eq:riso:balls}
      B(x,n) \;\subset\; \cS(x,n) \;\subset\; B(x, C_{\mathrm{W}} n)
    \end{align}
    and
    \begin{align}\label{eq:riso:weak}
      |\partial_{\cS(x, n)} A|
      \;\geq\;
      C_{\mathrm{riso}}\, n^{-1}\, |A|
    \end{align}
    for all $A \subset \cS(x, n)$ with $|A| \leq \frac{1}{2} |\cS(x, n)|$.
  \item[(iii)] isoperimetric inequality for large sets: there exists $N_3(x) < \infty$ such that for all $n \geq N_3(x)$,
    \begin{align}\label{eq:iso:large_sets}
      |\partial A|
      \;\geq\;
      C_{\mathrm{iso}}\, |A|^{(d-1)/d},
    \end{align}
    for all $A \subset B(x, n)$ with $|A| \geq n^{\th}$.
  \end{itemize}
\end{assumption}
\begin{remark}\label{rem:iso:mod}
  Suppose that a graph $G$ satisfies the relative isoperimetric inequality \eqref{eq:riso:weak} and $C_1 \in (0, 1/2)$.  Then, for all $n \geq N_2(x)$ and any $A \subset \cS(x,n)$ such that $\frac{1}{2} |\cS(x,n)| < |A| < (1-C_1) |\cS(x,n)|$, we have that
  \begin{align*}
    |\partial_{\cS(x, n)} A|
    \;=\;
    |\partial_{\cS(x, n)} (\cS(x, n) \setminus A)|
    \;\geq\;
    C_{\mathrm{riso}}\, n^{-1}\, |\cS(x, n) \setminus A|
    \;\geq\;
    C_{\mathrm{riso}}\, C_1\, n^{-1}\, |A|.
  \end{align*}
  Thus, any such set $A$ also satisfies the relative isoperimetric inequality however with a smaller constant.  
\end{remark}

\subsection{Sobolev inequality for functions with compact support}
By introducing an effective dimension, we first prove a (weak) isoperimetric inequality that holds for all subsets $A \subset B(x, n)$ provided that $n$ is large enough.
\begin{lemma}\label{lemma:iso:all_sets}
  Suppose that Assumption~\ref{ass:graph}~(i) and (iii) hold for some $\th \in [0, 1)$ and let $\zeta \in \big[0, \frac{1 - \th}{1 - \th/d}\big]$.  Then, for all $x \in V$ and $n \geq N_1(x) \vee N_3(x)$,
  \begin{align}\label{eq:iso:all_sets}
    \frac{|\partial A|}{|A|^{(d-\ze)/d}}
    \;\geq\;
    \frac{C_{\mathrm{iso}}/C_{\mathrm{reg}}^{(1 - \ze)/d} \wedge 1}{n^{1 - \ze}},
    \qquad \forall\, A \subset B(x, n).
  \end{align}
\end{lemma}
\begin{remark}
  Let $d \geq 2$ and $\ze \in [0, \frac{1 - \th}{1-\th/d}]$ for some $\th \in [0, 1)$.  By setting $d' \ldef d/\ze$, we have that $(d'-1)/d' = (d-\ze)/d$. Thus, \eqref{eq:iso:all_sets} corresponds to a (weak) isoperimetric inequality with $d$ replaced by $d'$.
\end{remark}
\begin{proof}
  Consider $\ze \in \big[0, (1 - \th)/(1 - \th/d)\big]$ and $x \in V$.  For any $n \geq N_1(x) \vee N_3(x)$, let $A \subset B(x, n)$ be non-empty.  In the sequel, we proceed by distinguish two cases: $|A|\geq n^{\th}$ and $|A|< n^{\th}$.  If $|A| \geq n^{\th}$, we have
  \begin{align*}
    \frac{|\partial A|}{|A|^{(d-\ze)/d}}
    \overset{\eqref{eq:iso:large_sets}}{\;\geq\;}
    \frac{C_{\mathrm{iso}}}{|A|^{(1-\ze)/d}}
    \;\geq\;
    \frac{C_{\mathrm{iso}}}{|B(x, n)|^{(1-\ze)/d}}
    \overset{\eqref{eq:reg:large_balls}}{\;\geq\;}
    \frac{C_{\mathrm{iso}}/C_{\mathrm{reg}}^{(1-\ze)/d}}{n^{1-\ze}}.
  \end{align*}
  On the other hand, in case $|A| < n^{\th}$, due to the choice of $\ze$ we obtain that
  \begin{align*}
    \frac{|\partial A|}{|A|^{(d-\ze)/d}}
    \;\geq\;
    \frac{1}{n^{\th(1 - \ze/d)}}
    \;\geq\;
    \frac{1}{n^{1-\ze}}.
  \end{align*}
  This completes the proof.
\end{proof}
\begin{prop}[Sobolev inequality]\label{prop:sobolev:S1}
  Suppose that Assumption~\ref{ass:graph}~(i) and (iii) hold for some $\th \in [0, 1)$.  Then, for any $\zeta \in \big[0, \frac{1 - \th}{1 - \th/d}\big]$, there exists $C_{\mathrm{S}_1} \equiv C_{\mathrm{S}_1}(\th, d) \in (0, \infty)$ such that, for any $x \in V$ and $n \geq N_1(x) \vee N_3(x)$,
  \begin{align}\label{eq:sobolev:S1}
    \Bigg(
      \frac{1}{|B(x, n)|}\, \sum_{y \in B(x, n)}\! |u(y)|^{\frac{d}{d - \ze}}
    \Bigg)^{\!\!\frac{d - \ze}{d}}
    \;\leq\;
    \frac{C_{\mathrm{S_1}} n}{|B(x, n)|}
    \sum_{\substack{y \vee y' \in B(x,n)\\ \{y, y'\} \in E}}\mspace{-6mu}
    \big|u(y) - u(y') \big|
  \end{align}
  for every function $u\! : V \to \bbR$ with $\supp u \subset B(x, n)$.
\end{prop}
\begin{proof}  
  The assertion follows by an application of the co-area formula as in \cite[Proposition~3.4]{Wo00}.  Nevertheless, we will repeat it here for the readers'\ convenience.

  Let $u\!: V \to \bbR$ be a function with $\supp u \subset B(x, n)$.  Note that it is enough to consider $u \geq 0$.  Moreover, by Jensen's inequality it suffices to prove \eqref{eq:sobolev:S1} for $\ze = \frac{1-\th}{1 - \th/d}$.  Define for $t \geq 0$ the super-level sets of $u$ by $A_t = \{y \in V : u(y) > t\}$.  Obviously, $A_t \subset B(x, n)$ for any $t \geq 0$.  Thus, 
  \begin{align*}
    \sum_{\{y,y'\} \in E}\! \big|u(y) - u(y')\big| 
    &\;=\; 
    \int_0^\infty\! |\partial A_{t}|\; \mathrm{d}t
    \overset{\!\eqref{eq:iso:all_sets}\!}{\;\geq\;}
    \frac{C_{\mathrm{iso}}/C_{\mathrm{reg}}^{(1-\ze)/d} \wedge 1}{n^{1 - \ze}}\,
    \int_0^\infty\! |A_{t}|^{1-\zeta/d}\; \mathrm{d}t.
  \end{align*}
  Now, consider a function $g\!: B(x_0, n) \to  [0, \infty)$ with $\norm{g}{\al_*}{V} = 1$ where $\al_* = d / \ze$ and $\al = d / (d-\ze)$.  Notice, that $1/\al + 1/\al_* = 1$ and $1 - \ze / d = 1/\al$.  Since $|A_{t}|^{1/\al} \geq \left\langle \indicator_{A_{t}}, g \right\rangle$ by H\"older's inequality, we obtain that
  \begin{align*}
    \int_0^\infty\! |A_{t}|^{1-\zeta/d}\; \mathrm{d}t
    \;\geq\;
    \int_0^\infty\! \left\langle \indicator_{A_{t}}, g \right\rangle\, \mathrm{d}t
    \;=\;
    \left\langle u, g \right\rangle.
  \end{align*}
  Here, $\langle \cdot, \cdot \rangle$ denotes the scalar product in $\ell^2(V)$.  Finally, taking the supremum over all $g\!: B(x_0, n) \to [0, \infty)$ with $\norm{g}{\al_*}{V} = 1$ implies the assertion \eqref{eq:sobolev:S1}.
\end{proof}
\begin{remark}
  It is well known, see \cite[Lemma 3.3.3]{Sa97}, that for functions $u\!: V \to \bbR$ that are not compactly supported, the following (weak) Poincar\'e inequality follows from the (weak) relative isoperimetric inequality: For $x \in V$ and $n \geq N_2(x)$
  \begin{align*}
    \inf_{m \in \bbR}\, \sum_{y \in B(x, n)}\! |u(y) - m|
    \;\leq\;
    C_{\mathrm{riso}}^{-1}\,n
    \sum_{\substack{y, y' \in B(x, C_\mathrm{W}n)\\ \{y, y'\} \in E}}\mspace{-6mu}
    \big|u(y) - u(y')\big|.
  \end{align*}
\end{remark}

\subsection{Anchored Sobolev inequality}
As a second result, we prove a Sobolev inequality for functions with unbounded support that vanishes at some point $x \in V$.  The proof is based on an \emph{anchored} relative isoperimetric inequality.  For this purpose, let $w\!: E \to (0, \infty)$ be an edge weight and for any $A \subset B$ non-empty we write
\begin{align*}
  |\partial_B A|_w
  \;\ldef\;
  \sum_{e \in \partial_B A} w(e).
\end{align*}
\begin{lemma}[anchored relative isoperimetric inequality]\label{lemma:riso:weight}
  Suppose that the graph $G$ satisfies Assumption~\ref{ass:graph}~(i) and (ii).  For any $x \in V$ and $\ve \in (0, 1)$, choose $n$ large enough such that $\lfloor n^{(1-\ve)/(d-\ve)}\rfloor \geq N_1(x) \vee N_2(x)$.  Further, assume that
  \begin{align}
    w_n(\{y,y'\})
    \;\geq\;
    (n / \max\{d(x, y), d(x, y')\})^{d-\ve}
    \qquad \forall\, \{y, y'\} \in E.
  \end{align}
  Then, there exists $C_2 \in (0, \infty)$ such that
  \begin{align}\label{eq:riso:weight}
    |\partial_{\cS(x, n)} A|_{w_n}
    \;\geq\;
    \frac{C_2}{n}\, |A|,
    \qquad \forall\, A \subset \cS(x, n) \setminus \{x\}.
  \end{align}
\end{lemma} 
\begin{remark}
  On the Euclidean lattice, $(\bbZ^d, E_d)$, the anchored relative isoperimetric inequality \eqref{eq:riso:weight} holds for all $n \geq 1$, if $w_n(\{y,y'\}) \geq c (n / \max\{d(x,y), d(x,y')\})^{d-1}$ some $c < \infty$ and for all $\{y,y\} \in E_d$.
\end{remark}
\begin{proof}
  Set $C_1 = 2^{-(d+1)} c_{\mathrm{reg}} C_{\mathrm{reg}}^{-1} C_{\mathrm{W}}^{-d}$ and $\be = (1-\ve)/(d-\ve)$.  In view of Remark~\ref{rem:iso:mod}, it holds that for any $n \geq N_1(x) \vee N_2(x)$,
  \begin{align}\label{eq:riso:all_sets:large}
    |\partial_{\cS(x, n)} A|
    \;\geq\;
    C_1 C_{\mathrm{riso}}\, n^{-1}\, |A|
    \;\rdef\;
    C_3\, n^{-1} |A|
  \end{align}
  for all $A \subset \cS(x, n)$ with $|A| \leq (1-C_1)|\cS(x, n)|$.  Suppose that $n$ is chosen in such a way that $\lfloor n^{\be} \rfloor \geq N_1(x) \vee N_2(x)$ and let $A \subset \cS(x, n)\setminus \{x\}$ be non-empty.  Since $\cS(x, n) \subset B(x, C_{\mathrm{W}} n)$, we have that $w_n(\{y, y'\}) \geq C_{\mathrm{W}}^{-d}$ and so
  \begin{align*}
    |\partial_{\cS(x, n)} A|_{w_n} \geq C_{\mathrm{W}}^{-d} |\partial_{\cS(x, n)} A|.
  \end{align*}
  Thus, \eqref{eq:riso:weight} follows from \eqref{eq:riso:all_sets:large} for any $A \subset \cS(x, n) \setminus \{x\}$ with $|A| \leq (1 - C_1)|\cS(x, n)|$.

  It remains to consider the case $|A| > (1-C_1)|\cS(x, n)|$.  We proceed by distinguishing two different cases.  First, assume that $A \cap \cS(x, \lfloor n^{\be} \rfloor) \ne \emptyset$.  Due to the fact that $A$ does not contain the vertex $x$, there exists at least one edge $\{y, y'\} \in E$ with $y \in A \subset \cS(x, \lfloor n^{\be} \rfloor)$ and $y' \in \cS(x, \lfloor n^{\be} \rfloor) \setminus A$.  This implies
  \begin{align}
    |\partial_{\cS(x, n)} A|_{w_n}
    &\;\geq\;
    \max\big\{ w_n(e) : e \in \partial_{\cS(x, n)} A \big\}
    \nonumber\\[.5ex]
    &\overset{\eqref{eq:riso:balls}}{\;\geq\;}
    \frac{n^{d-1}}{C_{\mathrm{W}}^d}\,
    \overset{\eqref{eq:reg:large_balls}}{\;\geq\;}
    \frac{C_{\mathrm{reg}}^{-1}}{C_{\mathrm{W}}^{2d}}\,n^{-1} |B(x, C_{\mathrm{W}} n)|
    \overset{\eqref{eq:riso:balls}}{\;\geq\;}
    \frac{C_{\mathrm{reg}}^{-1}}{C_{\mathrm{W}}^{2d}}\,n^{-1} |A|.
  \end{align}
  Consider now the case that $|A| > (1-C_1) |\cS(x, n)|$ and $A \cap \cS(x, \lfloor n^{\be} \rfloor) = \emptyset$.  Set
  \begin{align}\label{eq:def:k}
    k
    \;\ldef\;
    \min\big\{
      0 \leq j \leq n \;:\; |A \cap \cS(x, j)| > (1-C_1)\, |\cS(x, j)| 
    \big\}.
  \end{align}
  Obviously, $\lfloor n^{\be} \rfloor < k \leq n$.  Since $|A \cap \cS(x, k-1)| \leq (1-C_1)\, |\cS(x, k-1)|$ by definition of $k$, we obtain by exploiting the monotonicity of the sets $\cS(x, k)$ 
  \begin{align}
    |\partial_{\cS(x, n)} A|_{w_n}
    &\;\geq\;
    \big|\partial_{\cS(x, k-1)} \big(A \cap \cS(x, k-1)\big)\big|_{w_n}
    \nonumber\\[.5ex]
    &\;\geq\;
    \frac{1}{C_{\mathrm{W}}^d}\, \bigg(\frac{n}{k-1}\bigg)^{\!\!d-\ve}
    \big|\partial_{\cS(x, k-1)} \big(A \cap \cS(x, k-1)\big)\big|
    \nonumber\\[.5ex]
    &\overset{\!\eqref{eq:riso:all_sets:large}}{\;\geq\;}
    \frac{C_3}{C_{\mathrm{W}}^d}\,
    \frac{n^{d-\ve}}{k^{d+1-\ve}}\,
    |A \cap \cS(x, k-1)|.
    \label{eq:riso:est1}
  \end{align}
  On the other hand, since
  \begin{align*}
    |\cS(x, k-1)|
    &\overset{\eqref{eq:riso:balls}}{\;\geq\;}
    |B(x, k-1)|
    \\[.5ex]
    &\overset{\eqref{eq:reg:large_balls}}{\;\geq\;}
    \frac{c_{\mathrm{reg}}\, C_{\mathrm{reg}}^{-1}}{C_{\mathrm{W}}^d}\, 2^{-d}\,
    |B(x, C_{\mathrm{W}} k)|
    \;\geq\;
    \frac{c_{\mathrm{reg}}\,C_{\mathrm{reg}}^{-1}}{C_{\mathrm{W}}^d}\, 2^{-d}\,
    |\cS(x, k)|,
  \end{align*}
  we get that $|\cS(x, k) \setminus \cS(x, k-1)| \leq (1-2C_1) |\cS(x, k)|$.  Hence,
  \begin{align}
    |A \cap \cS(x, k-1)|
    &\overset{\!\eqref{eq:def:k}\!}{\;\geq\;}
    (1-C_1)\, |\cS(x, k)| - |\cS(x, k) \setminus \cS(x, k-1)|
    \;\geq\;
    C_1 c_{\mathrm{reg}}\, k^{d}.
    \label{eq:riso:est2}
  \end{align}
  By combining \eqref{eq:riso:est1} and \eqref{eq:riso:est2}, we find that
  \begin{align}
    |\partial_{\cS(x, n)} A|_{w_n}
    \overset{\eqref{eq:reg:large_balls}}{\;\geq\;}
    \frac{C_1 C_3 c_{\mathrm{reg}}}{C_{\mathrm{W}}^{2d} C_{\mathrm{reg}}}\, n^{-1}\,
    |B(x_0, C_{\mathrm{W}} n)|
    \;\geq\;
    \frac{C_1 C_3 c_{\mathrm{reg}}}{C_{\mathrm{W}}^{2d} C_{\mathrm{reg}}}\, n^{-1}\, 
    |A|.
  \end{align}
  By setting $C_2 \ldef \min\{C_3, C_{\mathrm{reg}}^{-1} C_{\mathrm{W}}^{-2d}, C_1 C_3 c_{\mathrm{reg}} C_{\mathrm{reg}}^{-1} C_{\mathrm{W}}^{-2d}\}$, the assertion \eqref{eq:riso:weight} follows.
\end{proof}
\begin{prop}[anchored Sobolev inequality]\label{prop:Sobolev1:weighted}
  Let $x \in V$ and suppose that the assumptions of Lemma~\ref{lemma:riso:weight} are satisfied.  Then, there exists $\overbar{C}_{\mathrm{S_1}} \in (0, \infty)$ such that
  \begin{align}\label{eq:Sobolev1:weighted}
    \sum_{y \in B(x, n)}\! |u(y)|
    \;\leq\;
    \overbar{C}_{\mathrm{S_1}}\, n\mspace{-12mu}
    \sum_{\substack{y, y' \in B(x, C_{\mathrm{W}} n)\\ \{y, y'\} \in E}}\mspace{-21mu}
    w_n(\{y, y'\})\, \big|u(y) - u(y') \big|
  \end{align}
  for every function $u\! : V \to \bbR$ with $u(x) = 0$.
\end{prop}
\begin{proof}
  The proof is based on an application of the co-area formula and the anchored relative isoperimetric inequality as derived in Lemma~\ref{lemma:riso:weight}.  For some $x \in V$, let $u\!: V \to \bbR$ be a function with $u(x) = 0$.  It suffices to consider $u \geq 0$.  Define for $t \geq 0$ the super-level sets of $u$ by $A_t = \{y \in \cS(x,n) : u(y) > t\}$.  Then,
  \begin{align*}
    \sum_{\substack{y, y' \in \cS(x,n) \\ \{y,y'\} \in E}}\mspace{-15mu} 
    w_n(\{y,y'\})\, \big|u(y) - u(y')\big| 
    &\;=\; 
    \int_0^\infty\! |\partial_{\cS(x,n)} A_{t}|_{w_n}\; \mathrm{d}t
    \\
    &\overset{\!\eqref{eq:riso:weight}\!}{\;\geq\;}
    \frac{C_2}{n}\,
    \int_0^\infty\! |A_{t}|\; \mathrm{d}t
    \;=\;
    \frac{C_2}{n}\,
    \sum_{y \in B(x,n)} |u(y)|.
  \end{align*}
  Since $\cS(x,n) \subset B(x, C_{\mathrm{W}}n)$ by Assumption~\ref{ass:graph}(ii), \eqref{eq:Sobolev1:weighted} follows.
\end{proof}

\appendix
\section{Ergodic theorem}
\label{appendix:ergodic}
In this appendix we provide an extension of the Birkhoff ergodic theorem that generalises the result obtained in \cite[Theorem 3]{BD03}.  Consider a probability space $(\Om, \cF, \prob)$ and a group of measure preserving transformations $\tau_x : \Om \to \Om$, $x \in \bbZ^d$ such that $\tau_{x+y}= \tau_x\circ \tau_y$.  Further, let $B_1 \ldef \{ x \in \bbR^d : |x| \leq 1 \}$. 
\begin{theorem}
  Let $\vp \in L^{1}(\prob)$ and $\ve \in (0, d)$.  Then, for $\prob$-a.e.\ $\om$,
  \begin{align}\label{eq:ergodic:conv}
    \lim_{n \to \infty}
    \frac{1}{n^d}
    \sideset{}{'}\sum_{x \in B(0, n)} \frac{\vp(\tau_x \om)}{|x/n|^{d-\ve}} 
    \;=\;
    \bigg( \int_{B_1} |x|^{-(d-\ve)}\, \md x \bigg)\, \mean\!\big[\vp\big],
  \end{align}
  where the summation is taken over all $x \in B(0, n) \setminus \{0\}$.
\end{theorem}
\begin{proof}
  To start with, notice that the ergodic theorem, see \cite[Theorem 3]{BD03}, implies that for $\prob$-a.e.\ $\om$
  \begin{align}\label{eq:lim:1}
    &\lim_{k \to \infty} \lim_{n \to \infty}\,
    \frac{1}{n^d}
    \sideset{}{'}\sum_{x \in B(0, n)}\!
    \Big( k \wedge |x/n|^{-(d-\ve)} \Big)\, \vp(\tau_x \om)
    \nonumber\\[.5ex]
    &\mspace{36mu}=\;
    \lim_{k \to \infty}
    \bigg( \int_{B_1} k \wedge |x|^{-(d-\ve)}\, \md x \bigg)\,
    \mean\!\big[\vp\big]
    \;=\;
    \bigg( \int_{B_1} |x|^{-(d-\ve)}\, \md x \bigg)\, \mean\!\big[\vp\big].
  \end{align}
  On the other hand, by means of Abel's summation formula, we have that
  \begin{align*}
    \bigg|
      \frac{1}{n^{\ve}}\,
      \sideset{}{'}\sum_{x \in B(0, n)}\! \frac{\vp(\tau_x \om)}{|x|^{d-\ve}}
    \bigg|
    &\;=\;
    \bigg|
      \frac{1}{n^{\ve}}\,
      \sum_{j=1}^{n}
      \frac{1}{j^{d-\ve}} \sum_{|x| = j} \vp(\tau_x \om)
    \bigg|
    \\[.5ex]
    &\;\leq\;
    \bigg|
      \frac{1}{n^{d}}
      \sideset{}{'}\sum_{x \in B(0,n)} \mspace{-6mu}\vp(\tau_x \om)
    \bigg|
    \,+\,
    \frac{d-\ve}{n^{\ve}}\,
    \sum_{j=1}^{n-1} \frac{1}{j^{1-\ve}}\,
    \bigg|
      \frac{1}{j^d}\!
      \sideset{}{'}\sum_{x \in B(0, j)} \mspace{-6mu}\vp(\tau_x \om)
    \bigg|,
  \end{align*}
  where we used that $j^{-(d-\ve)} - (j+1)^{-(d-\ve)} \leq (d-\ve) j^{-(d+1-\ve)}$.  From this estimate we deduce that for any $\vp \in L^{1}(\prob)$ and $\prob$-a.e.\ $\om$
  \begin{align}\label{eq:maxineq:weighted}
    \sup_{n \geq 1}\;
    \Bigg|\,
      \frac{1}{n^{\ve}}\mspace{-3mu}
      \sideset{}{'}\sum_{x \in B(0, n)} \mspace{-6mu}
      \frac{\vp(\tau_x \om)}{|x|^{d-\ve}}\,
    \Bigg|
    \;\leq\;
    \frac{C d}{\ve}\,
    \sup_{n \geq 0}\,
    \frac{1}{|B(0, n)|} 
    \sum_{x \in B(0, n)} \mspace{-6mu}\big|\vp(\tau_x \om)\big|
    \;<\;
    \infty,
  \end{align}
  where $C \ldef \sup_{n \geq 1} |B(0, n)| / n^d < \infty$.  On the other hand, for any $k \geq 1$
  \begin{align}\label{eq:sum:diff}
    &\frac{1}{n^d}\,
    \Bigg|
      \,\sideset{}{'}\sum_{x \in B(0,n)}\mspace{-9mu}
      \bigg(
        \frac{1}{|x/n|^{d-\ve}} - k \wedge \frac{1}{|x/n|^{d-\ve}}
      \bigg)\, \vp(\tau_x \om)
    \Bigg|
    \;\leq\;
    \frac{C}{k^{\ve/d}}\,
    \sup_{n \geq 1}\;
    \Bigg|\,
      \frac{1}{n^{\ve}}\mspace{-3mu}
      \sideset{}{'}\sum_{x \in B(0,n)} \mspace{-6mu}
      \frac{\vp(\tau_x \om)}{|x|^{d-\ve}}\,
    \Bigg|.
  \end{align}
  Since the last factor on the right-hand side of \eqref{eq:sum:diff} is finite due to \eqref{eq:maxineq:weighted}, we conclude that $\prob$-a.s.\
  \begin{align}\label{eq:lim:2}
    \lim_{k \to \infty}
    \frac{1}{n^d}\,
    \Bigg|
      \,\sideset{}{'}\sum_{x \in B(0, n)}
      \mspace{-6mu}|x/n|^{-(d-\ve)}\, \vp(\tau_x \om)
      \,-\!
      \sideset{}{'} \sum_{x \in B(0, n)}
      \mspace{-6mu}\Big( k \wedge |x/n|^{-(d-\ve)} \Big)\, \vp(\tau_x \om)\,
    \Bigg|
    \;=\;
    0
  \end{align}
  uniformly in $n$.  The assertion follows by combining \eqref{eq:lim:1} and \eqref{eq:lim:2}.
\end{proof}
\subsubsection*{Acknowledgement} We thank Martin Barlow for useful discussions and valuable comments and in particular for suggesting the interpolation argument that is used in Lemma~\ref{lemma:iso:all_sets}.   T.A.N. gratefully acknowledges financial support of the DFG Research Training Group (RTG 1845) ''Stochastic Analysis with Applications in Biology, Finance and Physics'' and the Berlin Mathematical School (BMS).

\bibliographystyle{plain}
\bibliography{literature}

\begin{thebibliography}{10}

\bibitem{ABDH13}
S.~Andres, M.~T. Barlow, J.-D. Deuschel, and B.~M. Hambly.
\newblock Invariance principle for the random conductance model.
\newblock {\em Probab. Theory Related Fields}, 156(3-4):535--580, 2013.

\bibitem{ADS15}
S.~Andres, J.-D. Deuschel, and M.~Slowik.
\newblock Invariance principle for the random conductance model in a degenerate
  ergodic environment.
\newblock {\em Ann. Probab.}, 43(4):1866--1891, 2015.

\bibitem{ADS16a}
S.~Andres, J.-D. Deuschel, and M.~Slowik.
\newblock Harnack inequalities on weighted graphs and some applications to the
  random conductance model.
\newblock {\em Probab. Theory Related Fields}, 164(3-4):931--977, 2016.

\bibitem{ADS16b}
S.~Andres, J.-D. Deuschel, and M.~Slowik.
\newblock Heat kernel estimates for random walks with degenerate weights.
\newblock {\em Electron. J. Probab.}, 21:Paper No. 33, 21, 2016.

\bibitem{BBT15}
M.~Barlow, K.~Burdzy, and {\'A}.~Tim{\'a}r.
\newblock Comparison of quenched and annealed invariance principles for random
  conductance model: {P}art {II}.
\newblock In Z.-Q. Chen, N.~Jacob, M.~Takeda, and T.~Uemura, editors, {\em
  Festschrift {M}asatoshi {F}ukushima}, volume~17 of {\em Interdiscip. Math.
  Sci.}, pages 59--83. World Sci. Publ., Hackensack, NJ, 2015.

\bibitem{BBT16}
M.~Barlow, K.~Burdzy, and {\'A}.~Tim{\'a}r.
\newblock Comparison of quenched and annealed invariance principles for random
  conductance model.
\newblock {\em Probab. Theory Related Fields}, 164(3-4):741--770, 2016.

\bibitem{Ba04}
M.~T. Barlow.
\newblock Random walks on supercritical percolation clusters.
\newblock {\em Ann. Probab.}, 32(4):3024--3084, 2004.

\bibitem{BD10}
M.~T. Barlow and J.-D. Deuschel.
\newblock Invariance principle for the random conductance model with unbounded
  conductances.
\newblock {\em Ann. Probab.}, 38(1):234--276, 2010.

\bibitem{BB07}
N.~Berger and M.~Biskup.
\newblock Quenched invariance principle for simple random walk on percolation
  clusters.
\newblock {\em Probab. Theory Related Fields}, 137(1-2):83--120, 2007.

\bibitem{Bi11}
M.~Biskup.
\newblock Recent progress on the random conductance model.
\newblock {\em Probab. Surv.}, 8:294--373, 2011.

\bibitem{BP07}
M.~Biskup and T.~M. Prescott.
\newblock Functional {CLT} for random walk among bounded random conductances.
\newblock {\em Electron. J. Probab.}, 12:no. 49, 1323--1348, 2007.

\bibitem{Bo93}
D.~Boivin.
\newblock Weak convergence for reversible random walks in a random environment.
\newblock {\em Ann. Probab.}, 21(3):1427--1440, 1993.

\bibitem{BD03}
D.~Boivin and J.~Depauw.
\newblock Spectral homogenization of reversible random walks on {$\Bbb Z^d$} in
  a random environment.
\newblock {\em Stochastic Process. Appl.}, 104(1):29--56, 2003.

\bibitem{BLM87}
J.~Bricmont, J.~L. Lebowitz, and C.~Maes.
\newblock Percolation in strongly correlated systems: the massless {G}aussian
  field.
\newblock {\em J. Statist. Phys.}, 48(5-6):1249--1268, 1987.

\bibitem{dMFGW89}
A.~De~Masi, P.~A. Ferrari, S.~Goldstein, and W.~D. Wick.
\newblock An invariance principle for reversible {M}arkov processes.
  {A}pplications to random motions in random environments.
\newblock {\em J. Statist. Phys.}, 55(3-4):787--855, 1989.

\bibitem{He82}
I.~S. Helland.
\newblock Central limit theorems for martingales with discrete or continuous
  time.
\newblock {\em Scand. J. Statist.}, 9(2):79--94, 1982.

\bibitem{KV86}
C.~Kipnis and S.~R.~S. Varadhan.
\newblock Central limit theorem for additive functionals of reversible {M}arkov
  processes and applications to simple exclusions.
\newblock {\em Comm. Math. Phys.}, 104(1):1--19, 1986.

\bibitem{Ko85}
S.~M. Kozlov.
\newblock The averaging method and walks in inhomogeneous environments.
\newblock {\em Uspekhi Mat. Nauk}, 40(2(242)):61--120, 238, 1985.

\bibitem{Ma08}
P.~Mathieu.
\newblock Quenched invariance principles for random walks with random
  conductances.
\newblock {\em J. Stat. Phys.}, 130(5):1025--1046, 2008.

\bibitem{MP07}
P.~Mathieu and A.~Piatnitski.
\newblock Quenched invariance principles for random walks on percolation
  clusters.
\newblock {\em Proc. R. Soc. Lond. Ser. A Math. Phys. Eng. Sci.},
  463(2085):2287--2307, 2007.

\bibitem{PRS13}
E.~B. Procaccia, Ron Rosenthal, and Art\"em Sapozhnikov.
\newblock Quenched invariance principle for simple random walk on clusters in
  correlated percolation models.
\newblock {\em Probab. Theory Related Fields}, 166(3-4):619--657, 2016.

\bibitem{RS13}
P.-F. Rodriguez and A.-S. Sznitman.
\newblock Phase transition and level-set percolation for the {G}aussian free
  field.
\newblock {\em Comm. Math. Phys.}, 320(2):571--601, 2013.

\bibitem{Sa97}
L.~Saloff-Coste.
\newblock Lectures on finite {M}arkov chains.
\newblock In P.~Bernard, editor, {\em Lectures on probability theory and
  statistics ({S}aint-{F}lour, 1996)}, volume 1665 of {\em Lecture Notes in
  Math.}, pages 301--413. Springer, Berlin, 1997.

\bibitem{Sa14}
A.~Sapozhnikov.
\newblock Random walks on infinite percolation clusters in models with
  long-range correlations.
\newblock {\em Preprint, available at arXiv:1410.0605}, 2014.

\bibitem{SS04}
V.~Sidoravicius and A.-S. Sznitman.
\newblock Quenched invariance principles for walks on clusters of percolation
  or among random conductances.
\newblock {\em Probab. Theory Related Fields}, 129(2):219--244, 2004.

\bibitem{Wo00}
W.~Woess.
\newblock {\em Random walks on infinite graphs and groups}, volume 138 of {\em
  Cambridge Tracts in Mathematics}.
\newblock Cambridge University Press, Cambridge, 2000.

\end{thebibliography}

\end{document}